\documentclass[11pt,english]{article}
\usepackage{courier}
\usepackage{pdfsync}
\usepackage[T1]{fontenc}
\usepackage[utf8]{inputenc}
\usepackage{csquotes}
\usepackage{geometry}
\usepackage{color}
\usepackage{babel}
\usepackage{verbatim}
\usepackage{refstyle}
\usepackage{lmodern}
\usepackage{textcomp}
\usepackage{mathrsfs}
\usepackage{mathtools}
\usepackage{amsmath}
\usepackage{amsthm}
\usepackage{amssymb}
\usepackage{esint}
\usepackage{microtype}

\usepackage[colorlinks=true,pdfpagemode=UseNone,urlcolor=blue,linkcolor=blue,citecolor=blue]{hyperref}
\makeatletter
\AtBeginDocument{\providecommand\eqref[1]{\ref{eq:#1}}}
\AtBeginDocument{\providecommand\figref[1]{\ref{fig:#1}}}
\AtBeginDocument{\providecommand\propref[1]{\ref{prop:#1}}}
\AtBeginDocument{\providecommand\secref[1]{\ref{sec:#1}}}
\AtBeginDocument{\providecommand\thmref[1]{\ref{thm:#1}}}
\AtBeginDocument{\providecommand\lemref[1]{\ref{lem:#1}}}
\AtBeginDocument{\providecommand\appendixref[1]{\ref{appendix:#1}}}
\AtBeginDocument{\providecommand\defref[1]{\ref{def:#1}}}
\RS@ifundefined{subsecref}
  {\newref{subsec}{name = \RSsectxt}}
  {}
\RS@ifundefined{thmref}
  {\def\RSthmtxt{theorem~}\newref{thm}{name = \RSthmtxt}}
  {}
\RS@ifundefined{lemref}
  {\def\RSlemtxt{lemma~}\newref{lem}{name = \RSlemtxt}}
  {}

\numberwithin{equation}{section}
\numberwithin{figure}{section}
\numberwithin{table}{section}
\theoremstyle{plain}
\newtheorem{thm}{\protect\theoremname}[section]
\theoremstyle{plain}
\newtheorem{lem}[thm]{\protect\lemmaname}
\theoremstyle{plain}
\newtheorem{prop}[thm]{\protect\propositionname}
\theoremstyle{definition}
\newtheorem{defn}[thm]{\protect\definitionname}

\usepackage{refstyle}
\newref{rem}{name=Remark~,names=Remarks~}
\newref{lem}{name=Lemma~,names=Lemmas~}
\newref{def}{name=Definition~,names=Definitions~}
\newref{thm}{name=Theorem~,names=Theorems~}
\newref{prop}{name=Proposition~,names=Propositions~}
\newref{cond}{name=Condition~,names=Conditions~}
\newref{cor}{name=Corollary~,names=Corollaries~}
\newref{sec}{name=Section~,names=Sections~}
\newref{fig}{name=Figure~,names=Figures~}
\newref{subsec}{name=Subsection~,names=Subsections~}
\newref{appendix}{name=Appendix~,names=Appendices~}
\newref{eq}{name=,refcmd=(\ref{#1})}
\usepackage{tikz}
\usepackage{pgfplots}
\pgfplotsset{compat=1.12}
\hypersetup{final}

\@ifundefined{showcaptionsetup}{}{\PassOptionsToPackage{caption=false}{subfig}}
\usepackage{subfig}
\makeatother

\providecommand{\definitionname}{Definition}
\providecommand{\lemmaname}{Lemma}
\providecommand{\propositionname}{Proposition}
\providecommand{\theoremname}{Theorem}

\newcommand{\be}{\begin{equation}}
\newcommand{\ee}{\end{equation}}
\newcommand{\B}{{\mathrm B}}
\newcommand{\T}{{\mathrm T}}

\newcommand{\cX}{{\mathcal X}}
\newcommand{\bal}{\begin{aligned}}
\newcommand{\enbal}{\end{aligned}}

\begin{document}
\global\long\def\supp{\operatorname{supp}}

\global\long\def\Uniform{\operatorname{Uniform}}

\global\long\def\dif{\mathrm{d}}

\global\long\def\e{\operatorname{e}}

\global\long\def\ii{\operatorname{i}}

\global\long\def\Cov{\operatorname{Cov}}

\global\long\def\Var{\operatorname{Var}}

\global\long\def\pv{\operatorname{p.v.}}

\global\long\def\e{\mathrm{e}}

\global\long\def\p{\mathrm{p}}

\global\long\def\R{\mathbf{R}}

\global\long\def\Law{\operatorname{Law}}

\global\long\def\supp{\operatorname{supp}}

\global\long\def\image{\operatorname{image}}

\global\long\def\dif{\mathrm{d}}

\global\long\def\eps{\varepsilon}

\global\long\def\sgn{\operatorname{sgn}}

\global\long\def\tr{\operatorname{tr}}

\global\long\def\Hess{\operatorname{Hess}}

\global\long\def\Re{\operatorname{Re}}

\global\long\def\Im{\operatorname{Im}}

\global\long\def\Dif{\operatorname{D}}

\global\long\def\divg{\operatorname{div}}

\newcommand{\email}[1]{{\href{mailto:#1}{\nolinkurl{#1}}}}
\title{Viscous shock solutions to the stochastic Burgers equation}
\author{Alexander Dunlap\thanks{Department of Mathematics, Courant Institute of Mathematical Sciences, New York University, New York, NY 10012 USA, \email{alexander.dunlap@cims.nyu.edu}} \and Lenya
Ryzhik\thanks{Department of Mathematics, Stanford University, Stanford, CA 94305 USA,  \email{ryzhik@stanford.edu}}}
\maketitle
\begin{abstract}
We define a notion of a viscous shock solution of the stochastic Burgers
equation that connects ``top'' and ``bottom'' spatially stationary
solutions of the same equation. Such shocks  generally
travel in space, but we show that they admit time-invariant
measures when viewed in their own reference frames. Under such a measure,
the viscous shock is a deterministic function of the bottom and top
solutions and the shock location. However, the measure of the bottom
and top solutions must be tilted  to account for the change of reference frame. We also show
a convergence result to these stationary shock solutions from solutions
initially connecting two constants, as time goes to infinity.
\end{abstract}

\section{Introduction}

We consider the one-dimensional stochastic Burgers equation, forced by the gradient of a Gaussian noise that is smooth in space and white in time:
\begin{equation}
\dif u(t,x)=\frac{1}{2}[\partial_{x}^{2}u(t,x)-\partial_{x}(u^{2})(t,x)]\dif t+\dif(\partial_{x}V)(t,x),\qquad t,x\in\mathbb{R}.\label{eq:SBE}
\end{equation}
Here, $V=\rho*W$, where $W$ is a cylindrical Wiener process on $L^{2}(\mathbb{R})$
whose covariance kernel is the identity, so the Itô time differential
$\dif W$ is a white noise on $\mathbb{R}\times\mathbb{R}$, and $\rho\in\mathcal{C}^{\infty}(\mathbb{R})\cap H^{1}(\mathbb{R})$.
We use~$*$ to denote spatial convolution. A detailed construction of the solutions
to \eqref{SBE} in a weighted space $\cX$ of continuous functions 
that grow at most as $|x|^{1/2+}$ at infinity an be found in~\cite{DGR19}.
We recall the precise result and the definition of this space  in \secref{stufffromthelastpaper}.

Spacetime-stationary solutions to the stochastic Burgers equation on the whole real line have been studied extensively
in the recent years. With apologies for the clumsiness, we will refer
to the single-time laws of such spacetime-stationary solutions as ``space-translation-invariant
invariant measures.'' Kick-type random forcing in \eqref{SBE} was considered in~\cite{Bak16,BCK14,BL19},
and the white in time setting, as in the present paper, was treated in \cite{DGR19}.  We also refer to these 
papers for references to the extensive literature on the torus case $x\in\mathbb{R}/\mathbb{Z}$.

For the unforced Burgers equation ($V\equiv0$ in \eqref{SBE}), spacetime-stationary
solutions are simply constants. In addition,
the unforced problem also admits   traveling wave solutions, known as viscous shocks,
that are perhaps of a more direct interest in applications than constant solutions.
They have the explicit form
\begin{equation}
u(t,x)=-a\tanh(a(x-bt-c))+b=\frac{b-a}{1+\e^{-2a(x-bt-c)}}+\frac{b+a}{1+\e^{2a(x-bt-c)}}\label{eq:deterministic-shocks}
\end{equation}
for constants $a,b,c\in\mathbb{R}$, as can be checked directly. Such
solutions ``connect'' the two constant solutions $b\pm|a|$, by
which we mean that
\[
\lim_{x\to\mp\infty}u(t,x)=b\pm|a|,
\]
as depicted in \figref{deterministicshock}.

\begin{figure}
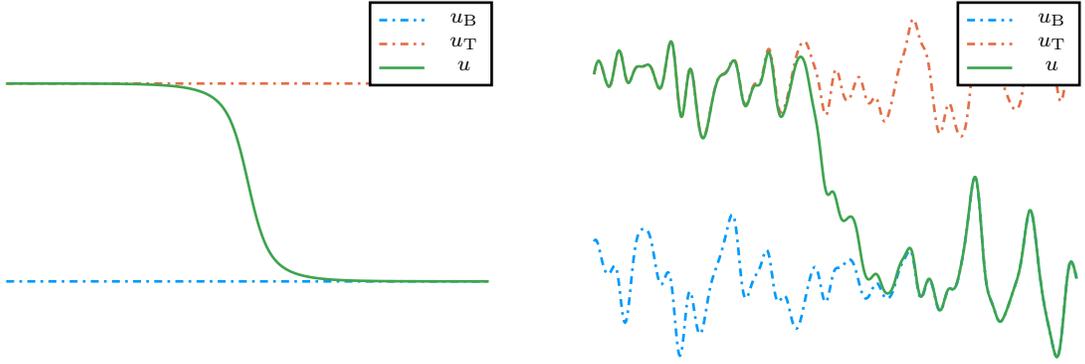

\centering{}\hfill{}\subfloat[\label{fig:deterministicshock}Viscous shock solution $u$ between
lower and upper solutions $u_{\mathrm{B}},u_{\mathrm{T}}$ of the
deterministic Burgers equation.]{

 
}\hfill{}\subfloat[\label{fig:randomshock}Viscous shock solution $u$ between lower
and upper solutions $u_{\mathrm{B}},u_{\mathrm{T}}$ of the stochastic
Burgers equation \eqref{SBE-many}.]{%
 
}\hspace*{\fill}\caption{Viscous shock solutions to the deterministic and stochastic Burgers equations.}
\end{figure}

\subsubsection*{Existence and classification of the random shock measures}

In this paper, we discuss analogues of these viscous shocks in the
stochastically forced case. We first must have analogues of the constant solutions that are connected by the shocks. We recall that \cite{DGR19} considered spacetime-stationary \emph{families}
of solutions to \eqref{SBE}, by which we mean jointly spacetime-stationary
solutions to the system of equations
\begin{equation}
\dif u_{i}(t,x)=\frac{1}{2}[\partial_{x}^{2}u_{i}(t,x)-\partial_{x}(u_{i}^{2})(t,x)]\dif t+\dif(\partial_{x}V)(t,x),\qquad t,x\in\mathbb{R},\label{eq:SBE-many}
\end{equation}
which are coupled through the noise $V$. 
As shown in~\cite[Theorem~1.1]{DGR19}, such coupled spacetime-stationary
solutions are almost-surely ordered according to their mean. One can
construct them as long-time limits of solutions starting with constant
initial conditions, and the limits preserve the order of the constants.
It was also shown in \cite[Theorem~1.1]{DGR19} that for any~$a_{k}\in\mathbb{R}$,
$k=1,\ldots,N$, there exists a unique extremal space-translation-invariant
invariant measure $\nu_{a_{1},\ldots,a_{N}}$ for \eqref{SBE-many}
such that if $(v_{1},\ldots,v_{N})\sim\nu_{a_{1},\ldots,a_{N}}$,
then $\mathbb{E}v_{k}(x)=a_{k}$ and $\mathbb{E}v_{k}(x)^{2}<\infty$
for each $x\in\mathbb{R}$ and $k=1,\ldots,N$. The extremal
invariant measures serve as attractors for the solutions to the Cauchy problem 
for a large class of ``not far from periodic'' initial conditions.

A stochastic shock, rather than connecting two constants as in the deterministic case,
connects two ordered components~$u_{\mathrm{B}}$ and
$u_{\mathrm{T}}$ (``bottom'' and ``top'') of a (space-stationary, say) solution to~\eqref{SBE-many} with $N=2$,
as illustrated in \figref{randomshock}.  We define the set of bottom and top solutions
\begin{equation}
\mathcal{X}_{\mathrm{BT}}=\left\{ (u_{\mathrm{B}},u_{\mathrm{T}})\in\mathcal{X}^{2}\ :\ u_{\mathrm{B}}<u_{\mathrm{T}}\text{ and }
\lim_{x\to\pm\infty}\int_{0}^{x}[u_{\mathrm{T}}-u_{\mathrm{B}}](y)\,\dif y=\pm\infty\right\}.\label{eq:XBTdefbis}
\end{equation}
The space of viscous shocks  is then
\begin{equation}
\mathcal{X}_{\mathrm{Sh}}=\left\{ (u_{\mathrm{B}},u_{\mathrm{T}},u)\in\mathcal{X}_{\mathrm{BT}}\times\mathcal{X}\ :\ 
\int_{-\infty}^{0}|u_{\mathrm{T}}-u|+\int_{0}^{\infty}|u-u_{\mathrm{B}}|<\infty\right\} .\label{eq:XShdefbis2}
\end{equation}
If $(u_\B,u_\T,u)\in\mathcal{X}_\mathrm{Sh}$, then we say that $u$ is a shock connecting $u_\T$ on the left to $u_\B$ on the right. We note that \eqref{XBTdefbis}--\eqref{XShdefbis2} give an ``$L^1$'' notion of a shock, which is convenient because of the nice $L^1$ properties of the stochastic Burgers dynamics (described in \cite[Section~3]{DGR19}).

Given a pair $(u_\B,u_\T)\in\cX_\mathrm{BT}$ of bottom and top solutions to \eqref{SBE}, one can construct 
a semi-explicit  shock solution to this equation
in terms of $u_{\mathrm{B}}$ and $u_{\mathrm{T}}$, generalizing
\eqref{deterministic-shocks}, so that the triple~$(u_\B,u_\T,u)$ lies in $\cX_{\mathrm {Sh}}$, as follows. 
If $v_{\mathrm{B}}(x)<v_{\mathrm{T}}(x)$ for all~$x\in\mathbb{R}$, and $b,\gamma\in\mathbb{R}$, define
\begin{equation}
\mathscr{S}_{b,\gamma}[v_{\mathrm{B}},v_{\mathrm{T}}](x)=\frac{v_{\mathrm{B}}(x)}{1+\exp\{\gamma-\int_{b}^{x}[v_{\mathrm{T}}-v_{\mathrm{B}}](y)\,\dif y\}}+\frac{v_{\mathrm{T}}(x)}{1+\exp\{-\gamma+\int_{b}^{x}[v_{\mathrm{T}}-v_{\mathrm{B}}](y)\,\dif y\}}.\label{eq:Sdef}
\end{equation}
Let $(u_{\mathrm{B}},u_{\mathrm{T}})$ be a solution to \eqref{SBE-many}
with $N=2$ such that $u_{\mathrm{B}}(t,x)<u_{\mathrm{T}}(t,x)$ for
all $t$ and $x$, and $b_t$ be the solution to the
non-autonomous ordinary differential equation
\begin{equation}
\partial_{t}b_{t}=\frac{1}{2}(-\partial_{x}(\log(u_{\mathrm{T}}-u_{\mathrm{B}}))+u_{\mathrm{B}}+u_{\mathrm{T}})(t,b_{t}).\label{eq:bteqn}
\end{equation}
If we set
\begin{equation}
u(t,x)=\mathscr{S}_{b_{t},\gamma}[(u_{\mathrm{B}},u_{\mathrm{T}})(t,\cdot)]\label{eq:shockprototype}
\end{equation}
for some fixed $\gamma\in\mathbb{R}$, then it turns out that $(u_{\mathrm{B}},u_{\mathrm{T}},u)$ solves \eqref{SBE-many}
with $N=3$. This is a general fact true for any pair of ordered solutions $u_\B$ and $u_\T$ of \eqref{SBE-many}.
We will refer to $b_{t}$ as the ``shock position.'' A more useful interpretation 
of $b_t$, in terms of the KPZ equation, is presented in \lemref{ZtsolvestheODE} in \secref{chgvar}.
We postpone it until then as it requires some additional notions.

If $(u_\B,u_\T)(t,\cdot)\in\cX_\mathrm{BT}$ (for which it suffices that this holds at $t=0$, as shown in \lemref{dynamicspreservelimits} below), 
and $u(t,x)$ is given by \eqref{shockprototype}, then for~$x-b_{t}\gg1$
we have $u\approx u_{\mathrm{B}}$, while for~$x-b_{t}\ll-1$ we have
$u\approx u_{\mathrm{T}}$. This is a direct way to see that \eqref{shockprototype} defines
a connection between~$u_\T$ on the left and $u_\B$ on the right.
The width of the transition region around~$b_{t}$ depends on the size of $u_{\mathrm{T}}-u_{\mathrm{B}}$ 
near~$b_{t}$: the closer $u_{\mathrm{T}}$ and~$u_{\mathrm{B}}$ get near~$b_{t}$, the wider the shock region. 
We will see this reflected in the tilt of the invariant measure in \thmref{stationary-shocks} below.

The system \eqref{bteqn}--\eqref{shockprototype}
involves the random noise $V$ only through $u_{\mathrm{B}}$ and
$u_{\mathrm{T}}$: conditional on the top and bottom solutions, the
shock  
position and profile are completely determined by~\eqref{bteqn} and~\eqref{shockprototype},
respectively.
The expression \eqref{shockprototype} is a direct generalization
of \eqref{deterministic-shocks}. Indeed, if $u_{\mathrm{B}}\equiv b-a$
and $u_{\mathrm{T}}=b+a$, with some $b\in\mathbb{R}$ and $a>0$,
then for any $b_{0}\in\mathbb{R}$, $b_{t}=bt+b_{0}$ solves \eqref{bteqn}.
Then~\eqref{shockprototype} reduces to~\eqref{deterministic-shocks},
with $c=b_{0}-\gamma/2$.

Motivated by \eqref{shockprototype} and continuing to assume the
ordering of $u_{\mathrm{B}}$ and $u_{\mathrm{T}}$, we can make a
change of variables
\begin{equation}
\zeta=\frac{1}{2}\int_{b_{t}}^{x}[u_{\mathrm{T}}-u_{\mathrm{B}}](t,x)\,\dif x,\qquad U=\frac{2u-u_{\mathrm{B}}-u_{\mathrm{T}}}{u_{\mathrm{T}}-u_{\mathrm{B}}}.\label{eq:chgvar-intro}
\end{equation}
Under this change of variables, \eqref{shockprototype} becomes the
deterministic and time-independent profile
\begin{equation}
U(t,\zeta)=-\tanh\zeta,\label{eq:Utzeta}
\end{equation}
which is the same as \eqref{deterministic-shocks} in the deterministic
case. As we will see, under the same change of variables, the stochastic
Burgers equation \eqref{SBE} takes the strikingly simple form
\begin{equation}
\partial_{t}U(t,\zeta)=\frac{1}{8}\partial_{\zeta}\left((u_{\mathrm{T}}-u_{\mathrm{B}})^{2}\cdot(\partial_{\zeta}U-U^{2}+1)\right)(t,\zeta),\label{eq:dtU-intro}
\end{equation}
to which \eqref{Utzeta} is a solution.

The above computations did not use any statistical properties of $u_{\mathrm{B}}$
and $u_{\mathrm{T}}$. Of particular interest  to us is the case when
$(u_{\mathrm{B}},u_{\mathrm{T}})$ is a spacetime-stationary solution
to \eqref{SBE-many} as constructed in~\cite[Theorem 1.2]{DGR19}.
Assume that $u_{\mathrm{B}}(t,x)<u_{\mathrm{T}}(t,x)$ for all $t$
and $x$ almost surely. In the deterministic case, the viscous shock
profile is stationary in the reference frame that moves with the constant
speed~$b$ of the shock. In the random case, the triple $(u_{\mathrm{B}},u_{\mathrm{T}},u)$
is not expected to be stationary in time, despite the time-stationarity
of the difference $u_{\mathrm{T}}-u_{\mathrm{B}}$ driving \eqref{dtU-intro},
because the shock location~$b_{t}$ need not be stationary. It is
natural to expect that $(u_{\mathrm{B}},u_{\mathrm{T}},u)$ would at
least be time-stationary in a reference frame moving along with $b_{t}$:
that is, that the randomly shifted triple~$\tau_{b_{0}-b_{t}}(u_{\mathrm{B}},u_{\mathrm{T}},u)$
would be time-stationary. Here, $\tau$ is the spatial translation
defined by
\be\label{aug3002}
\tau_{x}f(y)=f(y-x).
\ee 
This is not quite right either, because
$b_{t}$ is not independent of $(u_{\mathrm{B}},u_{\mathrm{T}})$.
We need to tilt the invariant measure to account for this dependence,
as described in the following theorem.
\begin{thm}
\label{thm:stationary-shocks}
Let $\nu$ be a space-translation-invariant
invariant measure for the dynamics \eqref{SBE-many} with $N=2$,
such that if $(v_{\mathrm{B}},v_{\mathrm{T}})\sim\nu$, then $\mathbb{E} v_{\mathrm{B}}(x)^2,\mathbb{E} v_{\mathrm{T}}(x)^2<\infty$ and $v_{\mathrm{B}}(x)<v_{\mathrm{T}}(x)$
for all $x\in\mathbb{R}$ almost surely. Fix $b\in\mathbb{R}$ and
define the measure $\hat{\nu}^{[b]}$, absolutely continuous with
respect to $\nu$, with Radon--Nikodym derivative
\begin{equation}
\frac{\dif\hat{\nu}^{[b]}}{\dif\nu}(v_{\mathrm{B}},v_{\mathrm{T}})=\frac{(v_{\mathrm{T}}-v_{\mathrm{B}})(b)}{\lim\limits _{L\to\infty}\frac{1}{L}\int_{0}^{L}[v_{\mathrm{T}}-v_{\mathrm{B}}](x)\,\dif x}.\label{eq:chgmeasure-intro}
\end{equation}
Fix $\gamma\in\mathbb{R}$ and
let~$(u_{\mathrm{B}},u_{\mathrm{T}},u)$ solve \eqref{SBE-many} with
initial condition $(u_{\mathrm{B}},u_{\mathrm{T}})\sim\hat{\nu}^{[b]}$,
independent of the noise, and
\[
u(0,x)=\mathscr{S}_{b,\gamma}[(u_{\mathrm{B}},u_{\mathrm{T}})(0,\cdot)].
\]
Let $b_{t}$ solve \eqref{bteqn} with $b_{0}=b$. Then for all $t\ge0$
we have
\begin{equation}
\Law(\tau_{b-b_{t}}(u_{\mathrm{B}},u_{\mathrm{T}},u)(t,\cdot))=\Law((u_{\mathrm{B}},u_{\mathrm{T}},u)(0,\cdot)).\label{eq:stationarity-intro}
\end{equation}
\end{thm}
Note that the limit in the denominator in~\eqref{chgmeasure-intro} exists $\nu$-almost surely
by the Birkhoff--Khinchin theorem. 

According to \cite[Theorem 1.2]{DGR19}, any space-translation-invariant
invariant measure for \eqref{SBE-many} having bounded second moment
can be decomposed into a mixture of extremal such measures, which are
classified by their means. As in \cite{DGR19}, we use the
notation $\nu_{a_{\mathrm{B}},a_{\mathrm{T}}}$ for the extremal
measure with mean $(a_{\mathrm{B}},a_{\mathrm{T}})$, and we
write $\hat{\nu}_{a_{\mathrm{B}},a_{\mathrm{T}}}^{[b]}$
for the tilt of this measure defined by \eqref{chgmeasure-intro}.
If~$(v_{\mathrm{B}},v_{\mathrm{T}})\sim\nu_{a_{\mathrm{B}},a_{\mathrm{T}}}$,
then \cite[Theorem 1.2, property (P5)]{DGR19} and the Birkhoff-Khinchin
theorem imply that
\[
\lim\limits _{L\to\infty}\frac{1}{L}\int_{0}^{L}[v_{\mathrm{T}}-v_{\mathrm{B}}](x)\,\dif x=a_{\mathrm{T}}-a_{\mathrm{B}},
\]
so in that case the change of measure formula \eqref{chgmeasure-intro} has the simpler form
\begin{equation}
\frac{\dif\hat{\nu}_{a_{\mathrm{B}},a_{\mathrm{T}}}^{[b]}}{\dif\nu_{a_\B,a_\T}}(v_{\mathrm{B}},v_{\mathrm{T}})=\frac{(v_{\mathrm{T}}-v_{\mathrm{B}})(b)}{a_{\mathrm{T}}-a_{\mathrm{B}}}.\label{eq:chgmeasure-intro-ergodicones}
\end{equation}

Note that \eqref{stationarity-intro} includes the statement that
$\Law(\tau_{b-b_{t}}(u_{\mathrm{B}},u_{\mathrm{T}})(t,\cdot))=\Law((u_{\mathrm{B}},u_{\mathrm{T}})(0,\cdot))$. In fact, this statement contains most of the content of \eqref{stationarity-intro} once the semi-explicit nature of the shock profiles is understood as in the discussion following \eqref{Sdef}. Nonetheless, the change of measure~\eqref{chgmeasure-intro} can be more easily understood
in the context of the shocks. (A more direct computational
reason for the tilt can be found in \lemref{ZtsolvestheODE} below, and in particular,
in expression \eqref{dtZinverse}.)
The tilt \eqref{chgmeasure-intro} is a type of size-biasing (or ``mass-biasing''), arising from the mass
conservation of both the Burgers dynamics and the change of variables
\eqref{chgvar-intro}. The Burgers dynamics \eqref{SBE-many}
has the form of a conservation law and so preserves the integrals
of differences between solutions (as is proven formally in \cite[Proposition 3.3]{DGR19}).
As we show in \propref{L1explicit}, we have
\be\label{eq:shintegralconserved}
\int_{\mathbb{R}}(\mathscr{S}_{b,\gamma}[v_{\mathrm{B}},v_{\mathrm{T}}]-\mathscr{S}_{b,\gamma'}[v_{\mathrm{B}},v_{\mathrm{T}}])=\gamma-\gamma'
\ee
for any $b,\gamma,\gamma'\in\mathbb{R}$ and $(v_{\mathrm{B}},v_{\mathrm{T}})$
in an appropriate function space.  
This is why $\gamma$ remains fixed in the evolution~\eqref{shockprototype}
and is thus a convenient way to parametrize the shocks. 
Consider the entire ensemble of shocks~$(\mathscr{S}_{b_{t},\gamma}[(u_{\mathrm{B}},u_{\mathrm{T}})(t,\cdot)])_{\gamma\in\mathbb{R}}$
evolving together between upper and lower solutions~$u_{\mathrm{B}}$
and $u_{\mathrm{T}}$. Now consider $\gamma,\gamma'\in\mathbb{R}$ with $|\gamma-\gamma'|\ll1$, and let $b_{t}^{(\gamma')}$ be a solution to \eqref{bteqn} (with a different initial condition than $b_t$) such that
\be\label{eq:Shbtgammaprime}
\mathscr{S}_{b_{t},\gamma'}[(u_{\mathrm{B}},u_{\mathrm{T}})(t,\cdot)]=\mathscr{S}_{b_{t}^{(\gamma')},\gamma}[(u_{\mathrm{B}},u_{\mathrm{T}})(t,\cdot)].
\ee
It follows from \eqref{shintegralconserved} and \eqref{Shbtgammaprime} that  
\begin{align*}
\gamma-\gamma' & =\int(\mathscr{S}_{b_{t},\gamma}[(u_{\mathrm{B}},u_{\mathrm{T}})(t,\cdot)]-\mathscr{S}_{b_{t}^{(\gamma')},\gamma}[(u_{\mathrm{B}},u_{\mathrm{T}})(t,\cdot)])\sim(b_{t}-b_{t}^{(\gamma')})\cdot(u_{\mathrm{T}}-u_{\mathrm{B}})(t,b_{t})
\end{align*}
is independent of $t$, and we must have 
\[
b_{t}-b_{t}^{(\gamma')}\sim\frac{\gamma-\gamma'}{(u_{\mathrm{T}}-u_{\mathrm{B}})(t,b_{t})}.
\]
This means that in an interval of size $\eps\ll1$ around $b_{t}$, we may
expect to find $b_{t}^{(\gamma')}$ for 
\[
|\gamma'-\gamma|\lesssim\eps(u_{\mathrm{T}}-u_{\mathrm{B}})(t,b_{t}),
\]
hence the change of measure \eqref{chgmeasure-intro}. 

One may ask about the uniqueness of the stationary shock profile measures
given by \eqref{chgmeasure-intro}. This question is not entirely
well-posed because one must specify the reference frame in which we
require stationarity.
We give a uniqueness
statement for the shock  in \propref{classifyssps} if the reference frame is assumed to be given by a shock location $\{b_t\}_{t\ge0}$ satisfying \eqref{bteqn}. The more intrinsic definition of~$b_t$
in \secref{chgvar} indicates that this choice of the reference frame is quite natural 
but further work is needed to understand uniqueness without fixing a particular reference frame.

Stationary shock behavior has been extensively studied for asymmetric
simple exclusion processes, which are discrete microscopic models
for Burgers-type dynamics. Similar phenomenology occurs there: a shock
moves randomly through space, but in the reference frame of the shock
itself, there is a stationary measure for the particle system \cite{DLS97,Fer92,FKS91}.
We refer to the book \cite{Lig99} for more discussion and references.

\subsubsection*{Stability of the random shocks}

We now turn to the stability of the shocks \eqref{shockprototype}.
The study of the stability of the shocks \eqref{deterministic-shocks}
in the deterministic case has a long history. Without any attempt
at completeness, we mention in particular the works \cite{FS98,Goo86,Hop50,IO58,IO60,JGK93,KV17,KM85,Nis85,OR82,Peg85,Pel71,Sat76}
and the books \cite{Daf16,Ser04}. In a similar spirit to our problem
is \cite{WX96}, which shows convergence to shock waves when the equation
is deterministic but the initial condition is random. As the Burgers
equation is nonlinear, these issues are closely related.

In the present stochastic setting, we show that if the initial condition is
sandwiched between two hyperbolic tangent functions (translated and
scaled appropriately), with the same limits at infinity, then an intermediate
solution, shifted appropriately, converges to a shock of the form
\eqref{shockprototype}. Actually, we show a somewhat stronger statement,
that if we consider a finite collection of such solutions, then they
converge jointly to a family of such shocks. In the following theorem, as above,
$\mathcal{X}$ denotes the Fréchet space of continuous functions on
$\mathbb{R}$ growing more slowly at infinity than $(1+|x|)^{\ell}$
for all $\ell>1/2$, equipped with the corresponding family of weighted seminorms
specified in \secref{stufffromthelastpaper}.
\begin{thm}
\label{thm:shock-stability}Fix real constants $a_{\mathrm{B}}<a_{\mathrm{T}}$ and
$\gamma_{\mathrm{L}}<\gamma_{\mathrm{R}}$. Let
$(u_{\mathrm{B}},u_{\mathrm{T}},u_{1},\ldots,u_{\mathrm{N}})\in\mathcal{C}([0,\infty);\mathcal{X}^{2+N})$
solve \eqref{SBE-many} with initial conditions $u_{\mathrm{B}}(0,\cdot)\equiv a_{\mathrm{B}}$,
$u_{\mathrm{T}}(0,\cdot)\equiv a_{\mathrm{T}}$, and for all $x\in\mathbb{R}$
and $i=1,\ldots,N$,
\[
\mathscr{S}_{0,\gamma_{\mathrm{L}}}[a_{\mathrm{B}},a_{\mathrm{T}}](x)\le u_{i}(0,x)\le\mathscr{S}_{0,\gamma_{\mathrm{R}}}[a_{\mathrm{B}},a_{\mathrm{T}}](x).
\]
For each $i=1,\ldots,N$, let $b^{(i)}$ be the unique $b$ so that
\begin{equation}
\int_{-\infty}^{b}[u_{\mathrm{T}}-u_{i}](0,x)\,\dif x=\int_{b}^{\infty}[u_{i}-u_{\mathrm{B}}](0,x)\,\dif x,\label{eq:setbi}
\end{equation}
and let $b_{t}^{(i)}$ solve \eqref{bteqn} with initial condition
$b_{0}^{(i)}=b^{(i)}$. Let $(v_{\mathrm{B}},v_{\mathrm{T}})\sim\hat{\nu}_{a_{\mathrm{B}},a_{\mathrm{T}}}^{[b^{(1)}]}$
(defined after \thmref{stationary-shocks}) and for $i=1,\ldots,N$,
put
\[
v_{i}=\mathscr{S}_{b^{(1)},b^{(i)}-b^{(1)}}[v_{\mathrm{B}},v_{\mathrm{T}}]
\]
and $\mathbf{v}=(v_{\mathrm{B}},v_{\mathrm{T}},v_{1},\ldots,v_{\mathrm{N}})$.
Then we have
\begin{equation}
\Law(\tau_{b^{(1)}-b_{t}^{(1)}}(u_{\mathrm{B}},u_{\mathrm{T}},u_{1},\ldots,u_{N})(t,\cdot))\to\Law\mathbf{v}\label{eq:convinlaw}
\end{equation}
weakly with respect to the topology of $\mathcal{X}^{2+N}$. Also,
with probability $1$ we have
\begin{equation}
\lim_{t\to\infty}\|u_{i}(t,\cdot)-\mathscr{S}_{b_{t}^{(i)},0}[(u_{\mathrm{B}},u_{\mathrm{T}})(t,\cdot)]\|_{L^{1}(\mathbb{R})}=0.\label{eq:trueL1conv}
\end{equation}
\end{thm}

We recall that, as far as the stability of $u_{\mathrm{B}}$ and $u_{\mathrm{T}}$
themselves is concerned, it was shown in~\cite[Theorem 1.3]{DGR19}
that if $\mathbf{u}\in\mathcal{C}([0,\infty);\mathcal{X}^{N})$ solves
\eqref{SBE-many} with initial condition that is a decaying perturbation of a spatially-periodic state 
$\mathbf{u}(0,\cdot)$,
then $\Law(\mathbf{u}(t,\cdot))$ converges to~$\nu_{a_{1},\ldots,a_{N}}$ weakly with respect to the topology of
$\mathcal{X}^{N}$ as $t\to\infty$. Even stronger results are available
for the stability of the spacetime-stationary solutions for the kick
forcing of the Burgers equation considered in \cite{BL19}. \thmref{shock-stability}, however, only considers the case when the top and bottom solutions are initially constant in space.

\subsubsection*{An interpretation of the shocks in terms of the Cole--Hopf transform}

The Burgers viscous shocks can be interpreted in terms of the Cole--Hopf
transform \cite{BG97,Col51,Hop50}. Recall that if $\phi$ solves
the multiplicative stochastic heat equation (SHE)
\begin{equation}
\dif\phi=\frac{1}{2}\partial_{x}^{2}\phi-\phi\dif V,\label{eq:SHE}
\end{equation}
then $h=-\log\phi$ solves the KPZ equation~\cite{KPZ86}
\begin{equation}
\dif h=\frac{1}{2}[\partial_{x}^{2}h-(\partial_{x}h)^{2}+\|\rho\|_{L^{2}(\mathbb{R})}^{2}]\dif t+\dif V,\label{eq:KPZ}
\end{equation}
and $u=\partial_{x}h=-(\partial_{x}\phi)/\phi$ solves the stochastic
Burgers equation \eqref{SBE}. Of course, this transform can be extended
to the system of equations \eqref{SBE-many}. The multiplicative SHE
\eqref{SHE} has the obvious advantage of being linear, but for our
purposes both \eqref{SHE} and \eqref{KPZ} have the disadvantage
that they do not admit spacetime-stationary solutions. Spacetime-stationary
solutions only arise when the derivative is taken to form $u$, which
destroys the growing zero-frequency mode of $h$.

Nonetheless, the stable viscous shock solutions \eqref{shockprototype}
have a simple interpretation in terms of solutions to the SHE \eqref{SHE}.
Indeed, if for $\mathrm{X}\in\{\mathrm{B},\mathrm{T}\}$, we have
$u_{\mathrm{X}}=-(\partial_{x}\phi_{\mathrm{X}})/\phi_{\mathrm{X}}$,
and $\phi_{\mathrm{X}}$ solves~\eqref{SHE}, then by linearity $\phi_{\mathrm{B}}+\phi_{\mathrm{T}}$
solves \eqref{SHE} as well, so that
\[
u=-\frac{\partial_{x}(\phi_{\mathrm{B}}+\phi_{\mathrm{T}})}{\phi_{\mathrm{B}}+\phi_{\mathrm{T}}}=\frac{u_{\mathrm{B}}}{1+\phi_{\mathrm{T}}/\phi_{\mathrm{B}}}+\frac{u_{\mathrm{T}}}{1+\phi_{\mathrm{B}}/\phi_{\mathrm{T}}}
\]
solves \eqref{SBE}. Noting that
\[
(\phi_{\mathrm{T}}/\phi_{\mathrm{B}})(t,x)=(\phi_{\mathrm{T}}/\phi_{\mathrm{B}})(t,0)\exp\left\{ -\int_{0}^{x}[u_{\mathrm{T}}-u_{\mathrm{B}}](t,y)\,\dif y\right\} ,
\]
we recover an expression of the form \eqref{shockprototype} by appropriate
choices of $\gamma$ and $b_{t}$.

Another, even more explicit, perspective considers the KPZ equation
in relation to the change of variables \eqref{chgvar-intro}. As we
show in \lemref{ZtsolvestheODE}, solutions to \eqref{bteqn} are
given by inverting (as a function from $\mathbb{R}\to\mathbb{R}$)
half the difference between two solutions to \eqref{KPZ}, started
at the corresponding integrals of the initial conditions for $u_{\mathrm{B}}$
and $u_{\mathrm{T}}$. Therefore, the integral appearing in the change
of variables \eqref{chgvar-intro} is exactly half the difference
of two solutions to \eqref{KPZ}. In addition, as shown in \lemref{ZtsolvestheODE},
the definition of the shock location $b_t$ is more naturally given in terms
of the solution to~\eqref{KPZ} than directly in terms of the Burgers equation itself.

Estimating the scale of fluctuations of $b_{t}$ is thus a question
about the growth of the difference between two solutions to \eqref{KPZ}.
Long-time statistics for solutions to \eqref{KPZ} are in general
difficult to estimate, especially in non-integrable cases such as
ours where exact calculations are not available. See \cite{BQS11,BCFV15,CH16,Nej20b,Nej20} and their references
for some results for integrable models, and \cite{BK18} for more
background and conjectures in this direction. We do not address the
question of estimating $b_{t}$ in the present paper, reserving it
for future work.

\paragraph{Organization of the paper}

We begin by introducing the relevant function spaces and recalling
the necessary setup and results from \cite{DGR19} in \secref{stufffromthelastpaper}.
We discuss the change of variables \eqref{chgvar-intro}, the resulting
PDE \eqref{dtU-intro}, and the explicit shock solutions \eqref{shockprototype}
in \secref{chgvar}. We derive the change of measure \eqref{chgmeasure-intro}
and prove \thmref{stationary-shocks} in \secref{BTsolns}. In \secref{uniquenessofshockprofiles},
we discuss more general shock profiles and give a partial characterization
of a certain notion of stationary shock profile (assuming some nontrivial
integrability conditions). Finally, we prove our stability result
\thmref{shock-stability} in \secref{stability}. A technical lemma
is relegated to \appendixref{techlemma}.

\paragraph{Acknowledgments}

We thank Erik Bates, Ivan Corwin, and Cole Graham for interesting
discussions. This work was supported
by NSF grants DGE-1147470, DMS-1613603, DMS-1910023, and DMS-2002118, BSF grant 2014302, and ONR grant N00014-17-1-2145.

\section{Function spaces and spacetime-stationary solutions\label{sec:stufffromthelastpaper}}

Because the viscous shock solutions to \eqref{SBE} are so intimately
tied to the spacetime-stationary solutions they connect, we rely on
the framework and many ingredients from \cite{DGR19}. Here, we review
the setup and quote some of the results we will use.

First, we recall some definitions and set the notation. For a positive
weight $w=w(x)$, we denote by $\mathcal{C}_{w}$ the Banach space
of continuous functions $f:\mathbb{R}\to\mathbb{R}$ such that the
norm
\[
\|f\|_{\mathcal{C}_{w}}=\sup_{x\in\mathbb{R}}\frac{|f(x)|}{w(x)}
\]
is finite. Given $\ell\in\mathbb{R}$, we set $\p_{\ell}=\langle x\rangle^{\ell}$,
where $\langle x\rangle=\sqrt{4+x^{2}}$, and let
\[
\mathcal{X}=\bigcap_{\ell>1/2}\mathcal{C}_{\p_{\ell}},
\]
equipped with the Fréchet space topology induced by the family of
norms $\{\|\cdot\|_{\mathcal{C}_{\p_{1/2+1/k}}}\}_{k\in\mathbb{N}}$.
This space is denoted by $\mathcal{X}_{1/2}$ in \cite{DGR19}. The
space $\mathcal{X}$ is separable and hence a Polish space.

The equation \eqref{SBE-many} is well-posed in $\mathcal{X}^{N}$,
as was proved in \cite[Theorem 1.1]{DGR19}. In particular, there
is a random solution map $\Psi:\mathcal{X}^{N}\to\mathcal{C}([0,\infty);\mathcal{X}^{N})$
for the equation \eqref{SBE-many}. The map $\Psi$ is almost surely
continuous with respect to the locally uniform topology on $\mathcal{C}([0,\infty);\mathcal{X}^{N})$.
It was also shown in \cite{DGR19} that \eqref{SBE-many} has a comparison
principle (\cite[Proposition 3.1]{DGR19}), and if the difference
of two components of a solution to \eqref{SBE-many} is in $L^{1}(\mathbb{R})$
at $t=0$, then its $L^{1}(\mathbb{R})$ norm is non-increasing in
time (\cite[Proposition 3.2]{DGR19}).

As we have mentioned, it is shown in \cite{DGR19} that for any given
set of means $a_{1},\dots,a_{N}$, there is a unique extremal space-translation-invariant
and \eqref{SBE-many}-invariant measure $\nu_{a_{1},\ldots,a_{N}}$
on $\mathcal{X}^{N}$ such that if $\mathbf{v}=(v_{1},\ldots,v_{N})\sim\nu_{a_{1},\ldots,a_{N}}$,
then $\mathbb{E}v_{i}(x)=a_{i}$ and $\mathbb{E}v_{i}(x)^{2}<\infty$
for all $x\in\mathbb{R}$. Here, ``extremal'' means that the measure
cannot be written as a nontrivial convex combination of measures with
the same properties.

In deriving properties of the shock solutions, it will be convenient
to state some necessary properties of the ``bottom'' and ``top''
spatially-stationary solutions in a nonprobabilistic way. We encode
these properties in the function space
\begin{equation}
\mathcal{X}_{\mathrm{BT}}=\left\{ (v_{\mathrm{B}},v_{\mathrm{T}})\in\mathcal{X}^{2}\ :\ v_{\mathrm{B}}<v_{\mathrm{T}}\text{ and }\lim_{x\to\pm\infty}\int_{0}^{x}[v_{\mathrm{T}}-v_{\mathrm{B}}](y)\,\dif y=\pm\infty\right\} ,\label{eq:XBTdef}
\end{equation}
as previously defined in \eqref{XBTdefbis}.
The conditions in \eqref{XBTdef} are necessary so that the change of variables~\eqref{chgvar-intro}
is invertible.

The next two lemmas in this section are technical in nature. \lemref{XBTPolish}
shows that $\mathcal{X}_{\mathrm{BT}}$ is a Polish space, so we can
apply standard probabilistic tools such as Prokhorov's theorem and
the Skorokhod representation theorem. \lemref{dynamicspreservelimits}
shows that the dynamics \eqref{SBE-many} preserves $\mathcal{X}_{\mathrm{BT}}$,
so we can think of solutions to \eqref{SBE-many} as Markov processes
on $\mathcal{X}_{\mathrm{BT}}$.
\begin{lem}
The space $\mathcal{X}_{\mathrm{BT}}$ is a Polish space.\label{lem:XBTPolish}
\end{lem}

\begin{proof}
We can write
\[
\{(v_{\mathrm{B}},v_{\mathrm{T}})\in\mathcal{X}^{2}\ :\ v_{\mathrm{B}}<v_{\mathrm{T}}\}=\bigcap_{L\in\mathbb{N}}\{(v_{\mathrm{B}},v_{\mathrm{T}})\in\mathcal{X}^{2}\ :\ v_{\mathrm{B}}(x)<v_{\mathrm{T}}(x)\text{ for all }x\in[-L,L]\}
\]
and
\begin{align*}
 & \left\{ (v_{\mathrm{B}},v_{\mathrm{T}})\in\mathcal{X}^{2}\ :\ \lim_{x\to\pm\infty}\int_{0}^{x}[v_{\mathrm{T}}-v_{\mathrm{B}}](y)\,\dif y=\pm\infty\right\} \\
 & \qquad=\bigcap_{M\in\mathbb{N}}\left(\bigcup_{L\in(0,\infty)}\left\{ (v_{\mathrm{B}},v_{\mathrm{T}})\in\mathcal{X}^{2}\ :\ \pm\int_{0}^{\pm L}[v_{\mathrm{T}}-v_{\mathrm{B}}](y)\,\dif y>M\right\} \right).
\end{align*}
Therefore, $\mathcal{X}_{\mathrm{BT}}$ is a countable
intersection of open subsets of $\mathcal{X}^{2}$, or in other words a~$G_{\delta}$ subset of the
Polish space $\mathcal{X}^{2}$. By Alexandrov's theorem (see e.g.
\cite[Theorem~2.2.1]{Sri98}) a $G_{\delta}$ subset of a Polish space
is again a Polish space.
\end{proof}
\begin{lem}
\label{lem:dynamicspreservelimits}If $\mathbf{u}\in\mathcal{C}([0,\infty);\mathcal{X}^{2+N})$
solves \eqref{SBE-many} with initial condition $\mathbf{u}(0,\cdot)\in\mathcal{X}_{\mathrm{BT}}\times\mathcal{X}^{N}$,
then with probability $1$ we have, for all $t\ge0$, that $\mathbf{u}(t,\cdot)\in\mathcal{X}_{\mathrm{BT}}\times\mathcal{X}^{N}$. 
\end{lem}

\begin{proof}
Let us write $\mathbf{u}=(\mathbf{u}_{\mathrm{BT}},\tilde{\mathbf{u}})=((u_{\mathrm{B}},u_{\mathrm{T}}),\tilde{\mathbf{u}})$.
The comparison principle (\cite[Theorem 3.1]{DGR19}) implies that,
with probability $1$, we have $u_{\mathrm{B}}(t,x)<u_{\mathrm{T}}(t,x)$
for all $t\ge0$ and all $x\in\mathbb{R}$. Thus it remains to prove
that, with probability $1$, we have for all $t\ge0$ that
\begin{equation}
\lim_{x\to\pm\infty}\int_{0}^{x}[u_{\mathrm{T}}-u_{\mathrm{B}}](t,y)\,\dif y=\pm\infty.\label{eq:limitsstayinfinity}
\end{equation}
We will prove that $+$ case of \eqref{limitsstayinfinity}; the $-$
case is analogous. The proof proceeds in a similar manner to that
of \cite[Proposition 3.3]{DGR19}. Fix $\ell>1/2$ and define
\[
\zeta(x)=\e^{2^{1-\ell}-\langle x\rangle^{1-\ell}},\qquad x\in\mathbb{R},
\]
and let $\chi$ be a smooth positive function so that $\chi|_{(-\infty,-1]}\equiv0$
and $\chi|_{[0,\infty)}\equiv1$. For $\delta>0$ we set
\[
\zeta_{\delta}(x)=\zeta(\delta x)\qquad\text{and}\qquad\omega_{\delta}(x)=\chi(x)\zeta_{\delta}(x).
\]
Then we have
\begin{align}
\frac{\dif}{\dif t}\int_{\mathbb{R}}[u_{\mathrm{T}}-u_{\mathrm{B}}](t,x)\omega_{\delta}(x)\,\dif x & =\frac{1}{2}\int_{\mathbb{R}}[\partial_{x}^{2}(u_{\mathrm{T}}-u_{\mathrm{B}})-\partial_{x}(u_{\mathrm{T}}^{2}-u_{\mathrm{B}}^{2})](t,x)\omega_{\delta}(x)\,\dif x\nonumber \\
 & =\frac{1}{2}\int_{\mathbb{R}}(u_{\mathrm{T}}-u_{\mathrm{B}})(t,x)[\omega_{\delta}''(x)+(u_{T}+u_{\mathrm{B}})(t,x)\omega_{\delta}'(x)]\,\dif x\label{eq:IBP}
\end{align}
using integration by parts. The boundary terms at infinity vanish
due to the at most polynomial growth of~$u_{\mathrm{B}}(t,x)$ and $u_{\mathrm{T}}(t,x)$
as $|x|\to\infty$ (\cite[Theorem 1.2, property (P4)]{DGR19}) and
the superpolynomial decay of $\zeta_{\delta}$ at infinity. Now, as
in \cite[(3.12)–(3.13)]{DGR19}, there is a constant $C<\infty$
(independent of $\delta$) so that
\[
|\zeta_{\delta}''(x)|\le C\delta^{2}\zeta_{\delta}(x)\qquad\text{and}\qquad\p_{\ell}(x)|\zeta_{\delta}'(x)|\le C\delta^{1-\ell}\zeta_{\delta}(x)\qquad\text{for all \ensuremath{x\in\mathbb{R}}.}
\]
Therefore, we have
\[
|\omega_{\delta}''(x)|\le C\delta^{2}\omega_{\delta}(x)\qquad\text{and}\qquad\p_{\ell}(x)|\omega_{\delta}'(x)|\le C\delta^{1-\ell}\omega_{\delta}(x)\qquad\text{for all \ensuremath{x\ge0}},
\]
and moreover (making $C$ larger if necessary)
\[
|\omega_{\delta}''(x)|,\p_{\ell}(x)|\omega_{\delta}'(x)|\le C\qquad\text{for all }x\in[-1,0].
\]
Using these bounds in \eqref{IBP}, we have
\begin{align}
 & \left|\frac{\dif}{\dif t}\int_{\mathbb{R}}[u_{\mathrm{T}}-u_{\mathrm{B}}](t,x)\omega_{\delta}(x)\,\dif x\right|\nonumber \\
 & \qquad\le\frac{1}{2}\int_{0}^{\infty}(u_{\mathrm{T}}-u_{\mathrm{B}})(t,x)[|\omega_{\delta}''(x)|+2\|\mathbf{u}_{\mathrm{BT}}(t,\cdot)\|_{\mathcal{C}_{\p_{\ell}}}\p_{\ell}(x)|\omega_{\delta}'(x)|]\,\dif x\nonumber \\
 & \qquad\qquad+C\int_{-1}^{0}(u_{\mathrm{T}}-u_{\mathrm{B}})(t,x)[1+\|\mathbf{u}_{\mathrm{BT}}(t,\cdot)\|_{\mathcal{C}_{\p_{\ell}}}]\,\dif x\nonumber \\
 & \qquad\le C(\delta^{2}+\|\mathbf{u}_{\mathrm{BT}}(t,\cdot)\|_{\mathcal{C}_{\p_{\ell}}}\delta^{1-\ell})\int_{0}^{\infty}(u_{\mathrm{T}}-u_{\mathrm{B}})(t,x)\omega_{\delta}(x)\,\dif x+C\langle\|\mathbf{u}_{\mathrm{BT}}(t,\cdot)\|_{\mathcal{C}_{\p_{\ell}}}\rangle^{2},\label{eq:derivbound}
\end{align}
where we have allowed the constant $C$ to change from line to line.
Now by the well-posedness proved in \cite[Theorem 1.1]{DGR19}, for
any $T\ge0$ we have 
\begin{equation}
\sup_{t\in[0,T]}\|\mathbf{u}_{\mathrm{BT}}(t,\cdot)\|_{\mathcal{C}_{\p_{\ell}}}<\infty\label{eq:uBTbound}
\end{equation}
almost surely. By the assumption that $\mathbf{u}_{\mathrm{BT}}(0,\cdot)\in\mathcal{X}_{\mathrm{BT}}$,
we have
\begin{equation}
\lim_{\delta\downarrow0}\int_{\mathbb{R}}[u_{\mathrm{T}}-u_{\mathrm{B}}](0,x)\omega_{\delta}(x)\,\dif x\to\infty.\label{eq:ictoinfinity}
\end{equation}
Combining \eqref{derivbound}, \eqref{uBTbound}, \eqref{ictoinfinity},
and Grönwall's inequality, we see that
\[
\lim_{\delta\downarrow0}\int_{\mathbb{R}}[u_{\mathrm{T}}-u_{\mathrm{B}}](t,x)\omega_{\delta}(x)\,\dif x=\infty,
\]
which implies that 
\[
\lim_{x\to\infty}\int_{0}^{x}[u_{\mathrm{T}}-u_{\mathrm{B}}](t,y)\,\dif y=\infty.\qedhere
\]
\end{proof}

\section{The change of variables and the explicit shock profiles\label{sec:chgvar}}

In this section we describe the change of variables \eqref{chgvar-intro}
leading to the equation \eqref{dtU-intro}, and show how this leads
to the explcit shock profiles \eqref{shockprototype}.

\subsection{The change of variables}

To understand the change of variables \eqref{chgvar-intro}, the first
step is to understand the shock position $b_{t}$ in terms of the
solution to the KPZ equation \eqref{KPZ}. Given a triple $\mathbf{v}=(v_{\mathrm{B}},v_{\mathrm{T}},v)\in\mathcal{X}_{\mathrm{BT}}\times\mathcal{X}$,
let~$\mathbf{u}=(u_{\mathrm{B}},u_{\mathrm{T}},u)\in\mathcal{C}([0,\infty);\mathcal{X}^{3})$
solve \eqref{SBE-many} with initial condition $\mathbf{u}(0,\cdot)=\mathbf{v}$.
By \lemref{dynamicspreservelimits}, we have with probability $1$
that
\begin{equation}
\mathbf{u}(t,\cdot)\in\mathcal{X}_{\mathrm{BT}}\times\mathcal{X}\label{eq:utinXBT}
\end{equation}
for all $t\ge0$. In addition, for a given $b\in\mathbb{R}$, let
$\mathbf{h}^{[b]}=(h_{\mathrm{B}}^{[b]},h_{\mathrm{T}}^{[b]},h^{[b]})$
solve the KPZ equation \eqref{KPZ} with initial condition
\[
\mathbf{h}^{[b]}(0,x)=\int_{b}^{x}\mathbf{u}(0,y)\,\dif y.
\]
We emphasize that $\mathbf{h}^{[b]}(t,x)$ is \emph{not} equal to
$\int_{b}^{x}\mathbf{u}(t,y)\,\dif y$, even though
\begin{equation}
\partial_{x}\mathbf{h}^{[b]}(t,x)=\mathbf{u}(t,x)\qquad\text{for all \ensuremath{t\ge0} and \ensuremath{x\in\mathbb{R}}}.\label{eq:partialxh}
\end{equation}
Indeed, we have
\begin{equation}
\mathbf{h}^{[b]}(t,x)=\mathbf{h}^{[b]}(t,b)+\int_{b}^{x}\mathbf{u}(t,y)\,\dif y,\label{eq:hTBFTC}
\end{equation}
and $\mathbf{h}^{[b]}(t,b)$ is not zero in general. Now we define
\begin{equation}
\overline{Z}_{b}[\mathbf{v}](x)=\frac{1}{2}\int_{b}^{x}[v_{\mathrm{T}}-v_{\mathrm{B}}](y)\,\dif y=\frac{1}{2}[h_{\mathrm{T}}^{[b]}-h_{\mathrm{B}}^{[b]}](0,x),\label{eq:Zbardef}
\end{equation}
and
\begin{equation}
Z_{b,t}[\mathbf{v}](x)=\frac{1}{2}[h_{\mathrm{T}}^{[b]}-h_{\mathrm{B}}^{[b]}](t,x).\label{eq:Zbtdef}
\end{equation}
Later on, we will use the notation $\overline{Z}_{b}[\mathbf{v}]$
and $Z_{b,t}[\mathbf{v}]$ when $\mathbf{v}\in\mathcal{X}_{\mathrm{BT}}\times\mathcal{X}^{N}$
for some $N\ge0$; the extension is obvious because these quantities
only depend on the first two components of $\mathbf{v}$. Observe
that $Z_{b,0}[\mathbf{v}](x)=\overline{Z}_{b}[\mathbf{v}](x)$ and
that for any $b'\in\mathbb{R}$ we have using \eqref{hTBFTC} that
\begin{equation}
Z_{b,t}[\mathbf{v}](x)=Z_{b,t}[\mathbf{v}](b')+\frac{1}{2}\int_{b'}^{x}[u_{\mathrm{T}}-u_{\mathrm{B}}](t,y)\,\dif y.\label{eq:Zaddition}
\end{equation}
By \eqref{utinXBT} and \eqref{XBTdef}, $Z_{b,t}[\mathbf{v}](x)$
is an invertible function of $x$ for each fixed $b$ and $t$. For
the rest of this section we will fix $\mathbf{v}$ and $\mathbf{u}$
as above, and write $Z_{b,t}=Z_{b,t}[\mathbf{v}]$.
\begin{lem}
\label{lem:ZtsolvestheODE}
Fix $b,\zeta\in\mathbb{R}$, and for $t\ge0$,
define $b_{t}=Z_{b,t}^{-1}(\zeta)$. Then $(b_{t})_{t\ge0}$ is the
unique solution to the ordinary differential equation \eqref{bteqn}
with initial condition $b_{0}=\overline{Z}_{b}[\mathbf{v}]^{-1}(\zeta)$. 
\end{lem}

\begin{proof}
We compute
\[
0=\partial_{t}\left(Z_{b,t}(Z_{b,t}^{-1}(\zeta))\right)=(\partial_{t}Z_{b,t})(Z_{b,t}^{-1}(\zeta))+(\partial_{x}Z_{b,t})(Z_{b,t}^{-1}(\zeta))\cdot\partial_{t}(Z_{b,t}^{-1}(\zeta).
\]
The fact that $(u_{\mathrm{B}},u_{\mathrm{T}})(t,\cdot)\in\mathcal{X}_{\mathrm{BT}}$
means that $\partial_{x}Z_{b,t}(x)\ne0$ for all $x\in\mathbb{R}$,
so we can compute
\begin{align}
\partial_{t}(Z_{b,t}^{-1})(\zeta) & =-\left(\frac{\partial_{t}Z_{b,t}}{\partial_{x}Z_{b,t}}\right)(Z_{b,t}^{-1}(\zeta))=-\frac{1}{2}\left(\frac{\partial_{x}(u_{\mathrm{T}}-u_{\mathrm{B}})-u_{\mathrm{T}}^{2}+u_{\mathrm{B}}^{2}}{u_{\mathrm{T}}-u_{\mathrm{B}}}\right)(t,Z_{b,t}^{-1}(\zeta))\nonumber \\
 & =\frac{1}{2}\left(-\partial_{x}(\log(u_{\mathrm{T}}-u_{\mathrm{B}}))+u_{\mathrm{B}}+u_{\mathrm{T}}\right)(t,Z_{b,t}^{-1}(\zeta)),\label{eq:ourvectorfield}
\end{align}
so $t\mapsto Z_{b,t}^{-1}(\zeta)$ satisfies \eqref{bteqn} for any
fixed $b$ and $\zeta$. The vector field on the right side of \eqref{ourvectorfield}
is locally Lipschitz, so the uniqueness comes from the basic theory
of ordinary differential equations.
\end{proof}
Note that the initial condition $b_{0}$ does not determine $b$ and
$\zeta$ uniquely. However, if we fix some~$\zeta\in\mathbb{R}$,
then $b$ is determined uniquely by $b_{0}$. In particular, if $\zeta=0$
then $b=b_{0}$. Alternatively, if we fix $b$, which determines $h_{\mathrm{B}}^{[b]}$
and $h_{\mathrm{T}}^{[b]}$, then the choice of $b_{0}$ is equivalent
to the choice of $\zeta$, as
\[
\zeta=\frac{1}{2}[h_{\mathrm{T}}^{[b]}-h_{\mathrm{B}}^{[b]}](0,b_{0}),
\]
and then the solution $(b_{t})_{t\ge0}$ to \eqref{bteqn} with initial
condition $b_{0}$ is determined by the condition that
\[
\zeta=\frac{1}{2}[h_{\mathrm{T}}^{[b]}-h_{\mathrm{B}}^{[b]}](t,b_{t}).
\]
This gives a very simple geometric interpretation of $b_{t}$ in terms
of the graphs of $h_{\mathrm{B}}^{[b]}(t,\cdot)$ and $h_{\mathrm{T}}^{[b]}(t,\cdot)$.

With this notation introduced, we see that the change of variables
\eqref{chgvar-intro} becomes
\begin{equation}
\zeta=Z_{b_{0},t}(x),\qquad U=\frac{2u-u_{\mathrm{T}}-u_{\mathrm{B}}}{u_{\mathrm{T}}-u_{\mathrm{B}}}.\label{eq:chgvar}
\end{equation}
The inverse change of variables is
\begin{equation}
x=Z_{b_{0},t}^{-1}(\zeta),\qquad u=\frac{1}{2}[(u_{\mathrm{T}}-u_{\mathrm{B}})U+u_{\mathrm{T}}+u_{\mathrm{B}}].\label{eq:inversechgvar}
\end{equation}
A convenient way to carry out this change of variables is to first
define the corresponding KPZ object
\begin{equation}
Q^{[b_{0}]}(t,\zeta)=\left(h^{[b_{0}]}-\frac{1}{2}h_{\mathrm{T}}^{[b_{0}]}-\frac{1}{2}h_{\mathrm{B}}^{[b_{0}]}\right)(t,Z_{b_{0},t}^{-1}(\zeta)),\label{eq:Qbdef}
\end{equation}
and then put
\begin{equation}
U^{[b_{0}]}(t,\zeta)=\partial_{\zeta}Q^{[b_{0}]}(t,\zeta)=\left(\frac{2u-u_{\mathrm{T}}-u_{\mathrm{B}}}{u_{\mathrm{T}}-u_{\mathrm{B}}}\right)(t,Z_{b_{0},t}^{-1}(\zeta)).\label{eq:Ubdef}
\end{equation}
In \eqref{Ubdef} we used the fact that
\begin{equation}
\partial_{\zeta}Z_{b_{0},t}^{-1}(\zeta)=\frac{1}{(\partial_{x}Z_{b_{0},t})(Z_{b_{0},t}^{-1}(\zeta))}=\frac{2}{(u_{\mathrm{T}}-u_{\mathrm{B}})(t,Z_{b_{0},t}^{-1}(\zeta))}\label{eq:dtZinverse}
\end{equation}
by \eqref{partialxh}.

Having carried out the change of variables, we now show that $U^{[b_{0}]}$
solves the PDE \eqref{dtU-intro}.
\begin{prop}
We have
\begin{equation}
\partial_{t}U^{[b_{0}]}(t,\zeta)=\frac{1}{8}\partial_{\zeta}\left(J^{[b_{0}]}(\zeta)\left(\partial_{\zeta}U^{[b_{0}]}(t,\zeta)-(U^{[b_{0}]}(t,\zeta))^{2}+1\right)\right),\label{eq:dtU}
\end{equation}
where
\[
J^{[b_{0}]}(\zeta)=(u_{\mathrm{T}}-u_{\mathrm{B}})^{2}(t,Z_{b_{0},t}^{-1}(\zeta)).
\]
\end{prop}

\begin{proof}
We start by computing a PDE for $Q^{[b_{0}]}$. Using \eqref{partialxh}
and \eqref{ourvectorfield}, we can differentiate \eqref{Qbdef} to
obtain
\begin{equation}
\begin{aligned}\partial_{t}Q^{[b_{0}]}(t,\zeta) & =\frac{1}{2}\left(\partial_{x}\left(u-\frac{1}{2}u_{\mathrm{T}}-\frac{1}{2}u_{\mathrm{B}}\right)-u^{2}+\frac{1}{2}u_{\mathrm{T}}^{2}+\frac{1}{2}u_{\mathrm{B}}^{2}\right)(t,Z_{b_{0},t}^{-1}(\zeta))\\
 & \qquad-\frac{1}{2}\left(\left(u-\frac{1}{2}u_{\mathrm{T}}-\frac{1}{2}u_{\mathrm{B}}\right)\cdot\left((\partial_{x}(\log(u_{\mathrm{T}}-u_{\mathrm{B}}))-(u_{\mathrm{B}}+u_{\mathrm{T}})\right)\right)(t,Z_{b_{0},t}^{-1}(\zeta)).
\end{aligned}
\label{eq:dtQ}
\end{equation}
On the other hand, we can differentiate the second equality in \eqref{Ubdef}
(using \eqref{dtZinverse} again) to get
\begin{align}
\partial_{\zeta}^{2} & Q^{[b_{0}]}(t,\zeta)\nonumber \\
 & =\frac{4}{(u_{\mathrm{T}}-u_{\mathrm{B}})^{2}}\left(\partial_{x}\left(u-\frac{1}{2}u_{\mathrm{T}}-\frac{1}{2}u_{\mathrm{B}}\right)-\left(u-\frac{1}{2}u_{\mathrm{T}}-\frac{1}{2}u_{\mathrm{B}}\right)\partial_{x}(\log(u_{\mathrm{T}}-u_{\mathrm{B}}))\right)(t,Z_{b_{0},t}^{-1}(\zeta)).\label{eq:dzeta2Q}
\end{align}
Recognizing the two terms in brackets in \eqref{dzeta2Q} in \eqref{dtQ},
we see that
\begin{align}
\partial_{t}Q^{[b_{0}]}(t,\zeta) & =\frac{1}{8}(u_{\mathrm{T}}-u_{\mathrm{B}})^{2}(t,Z_{b_{0},t}^{-1}(\zeta))\cdot\partial_{\zeta}^{2}Q^{[b_{0}]}(t,\zeta)\nonumber \\
 & \qquad+\frac{1}{2}\left(\frac{1}{2}u_{\mathrm{T}}^{2}+\frac{1}{2}u_{\mathrm{B}}^{2}-u^{2}+\left(u-\frac{1}{2}u_{\mathrm{T}}-\frac{1}{2}u_{\mathrm{B}}\right)(u_{\mathrm{B}}+u_{\mathrm{T}})\right)(t,Z_{b_{0},t}^{-1}(\zeta))\nonumber \\
 & =\frac{1}{8}(u_{\mathrm{T}}-u_{\mathrm{B}})^{2}(t,Z_{b_{0},t}^{-1}(\zeta))\cdot\partial_{\zeta}^{2}Q^{[b_{0}]}(t,\zeta)+\frac{1}{2}((u_{\mathrm{T}}-u)(u-u_{\mathrm{B}}))(t,Z_{b_{0},t}^{-1}(\zeta))\nonumber \\
 & =\frac{1}{8}(u_{\mathrm{T}}-u_{\mathrm{B}})^{2}(t,Z_{b_{0},t}^{-1}(\zeta))\cdot\left[\partial_{\zeta}^{2}Q^{[b_{0}]}(t,\zeta)-\left(\partial_{\zeta}Q^{[b_{0}]}(t,\zeta)\right)^{2}+1\right].\label{eq:dtQeqn}
\end{align}
Differentiating \eqref{dtQeqn} in $\zeta$ and recalling \eqref{Ubdef},
we get
\[
\partial_{t}U^{[b_{0}]}(t,\zeta)=\frac{1}{8}\partial_{\zeta}\left((u_{\mathrm{T}}-u_{\mathrm{B}})^{2}(t,Z_{b_{0},t}^{-1}(\zeta))\cdot\left[\partial_{\zeta}U^{[b_{0}]}(t,\zeta)-\left(U^{[b_{0}]}(t,\zeta)\right)^{2}+1\right]\right),
\]
which is \eqref{dtU}.
\end{proof}

\subsection{The shock profiles}

We now describe the explicit shock profiles introduced in \eqref{shockprototype}.
It is clear from \eqref{dtU} that, for any~$\gamma\in\mathbb{R}$,
the deterministic profile
\[
U_{\gamma}(t,\zeta)=-\tanh(\zeta-\gamma/2)
\]
is a solution to \eqref{dtU}. Applying the change of variables \eqref{inversechgvar},
we see that if we define
\begin{align}
u^{[b_{0},\gamma]}(t,x) & =\frac{1}{2}[-(u_{\mathrm{T}}-u_{\mathrm{B}})(t,x)\tanh(Z_{b_{0},t}(x)-\gamma/2)+(u_{\mathrm{B}}+u_{\mathrm{T}})(t,x)]\nonumber \\
 & =\frac{\e^{Z_{b_{0},t}(x)-\gamma/2}}{\e^{Z_{b_{0},t}(x)-\gamma/2}+\e^{-Z_{b_{0},t}(x)+\gamma/2}}u_{\mathrm{B}}(t,x)+\frac{\e^{-Z_{b_{0},t}(x)+\gamma/2}}{\e^{Z_{b_{0},t}(x)-\gamma/2}+\e^{-Z_{b_{0},t}(x)+\gamma/2}}u_{\mathrm{T}}(t,x)\nonumber \\
 & =\frac{1}{1+\e^{\gamma-2Z_{b_{0},t}(x)}}u_{\mathrm{B}}(t,x)+\frac{1}{1+\e^{2Z_{b_{0},t}(x)-\gamma}}u_{\mathrm{T}}(t,x),\label{eq:ub0gamma}
\end{align}
then $(u_{\mathrm{B}},u_{\mathrm{T}},u^{[b_{0},\gamma]})$ solves
\eqref{SBE-many}.

We note (recalling the definitions \eqref{Sdef} and \eqref{Zbardef})
that
\begin{equation}
\mathscr{S}_{b,\gamma}[v_{\mathrm{B}},v_{\mathrm{T}}](x)=\frac{v_{\mathrm{B}}(x)}{1+\e^{\gamma-2\overline{Z}_{b}[v_{\mathrm{B}},v_{\mathrm{T}}](x)}}+\frac{v_{\mathrm{T}}(x)}{1+\e^{2\overline{Z}_{b}[v_{\mathrm{B}},v_{\mathrm{T}}](x)-\gamma}}.\label{eq:Sbgammarewrite}
\end{equation}
Using \eqref{Zaddition} with $b=b_{0}$ and $b'=b_{t}$, and noting
by \lemref{ZtsolvestheODE} (with $\zeta=0$) that $Z_{b_{0},t}(b_{t})=0$,
we have
\begin{align*}
Z_{b_{0},t}(x) & =Z_{b_{0},t}(b_{t})+\frac{1}{2}\int_{b_{t}}^{x}[u_{\mathrm{T}}-u_{\mathrm{B}}](t,y)\,\dif y\\
 & =\frac{1}{2}\int_{b_{t}}^{x}[u_{\mathrm{T}}-u_{\mathrm{B}}](t,y)\,\dif y=\overline{Z}_{b_{t}}[(u_{\mathrm{B}},u_{\mathrm{T}})(t,\cdot)](x).
\end{align*}
Substituting this into \eqref{ub0gamma} and using \eqref{Sbgammarewrite},
we get
\begin{equation}
u^{[b_{0},\gamma]}(t,x)=\frac{u_{\mathrm{B}}(t,x)}{1+\e^{\gamma-2\overline{Z}_{b_{t}}[(u_{\mathrm{B}},u_{\mathrm{T}})(t,\cdot)](x)}}+\frac{u_{\mathrm{T}}(t,x)}{1+\e^{2\overline{Z}_{b_{t}}[(u_{\mathrm{B}},u_{\mathrm{T}})(t,\cdot)](x)-\gamma}}=\mathscr{S}_{b_{t},\gamma}[(u_{\mathrm{B}},u_{\mathrm{T}})(t,\cdot)](x).\label{eq:uisS}
\end{equation}

Let us record the $L^{1}(\mathbb{R})$ distances between two of these
explicit shock profiles.
\begin{prop}
\label{prop:L1explicit}If $\mathbf{v}_{\mathrm{BT}}=(v_{\mathrm{B}},v_{\mathrm{T}})\in\mathcal{X}_{\mathrm{BT}}$,
then
\begin{equation}
\|\mathscr{S}_{b_{0},\gamma}[\mathbf{v}_{\mathrm{BT}}]-\mathscr{S}_{b_{0},\gamma'}[\mathbf{v}_{\mathrm{BT}}]\|_{L^{1}(\mathbb{R})}=\left|\int_{\mathbb{R}}\left(\mathscr{S}_{b_{0},\gamma}[\mathbf{v}_{\mathrm{BT}}]-\mathscr{S}_{b_{0},\gamma'}[\mathbf{v}_{\mathrm{BT}}]\right)\right|=|\gamma-\gamma'|.\label{eq:Sdiffint}
\end{equation}
\end{prop}

\begin{proof}
It is clear from the definition \eqref{Sdef} that $\mathscr{S}_{b_{0},\gamma}[\mathbf{v}_{\mathrm{BT}}]$
and $\mathscr{S}_{b_{0},\gamma'}[\mathbf{v}_{\mathrm{BT}}]$ are ordered,
hence the first equality. For the second equality, we note that the
change of variables \eqref{chgvar} (with $t=0$) can be written as
\[
U(t,\zeta)=u(t,\overline{Z}_{b}[\mathbf{v}_{\mathrm{BT}}]^{-1}(\zeta))\partial_{\zeta}\overline{Z}_{b}[\mathbf{v}_{\mathrm{BT}}]^{-1}(\zeta)-\frac{u_{\mathrm{B}}+u_{\mathrm{T}}}{u_{\mathrm{T}}-u_{\mathrm{B}}}(t,\overline{Z}_{b}[\mathbf{v}_{\mathrm{BT}}]^{-1}(\zeta)),
\]
hence the integral in \eqref{Sdiffint} becomes
\[
\int_{\mathbb{R}}|-\tanh(\zeta-\gamma/2)-(-\tanh(\zeta-\gamma'/2))|\,\dif\zeta=\gamma-\gamma'.\qedhere
\]
\end{proof}
\begin{prop}
\label{prop:Scts}The map $\mathbb{R}^{2}\times\mathcal{X}_{\mathrm{BT}}\ni((b,\gamma),(v_{\mathrm{B}},v_{\mathrm{T}}))\mapsto\mathscr{S}_{b,\gamma}[v_{\mathrm{B}},v_{\mathrm{T}}]\in\mathcal{X}$
is continuous.
\end{prop}

\begin{proof}
Suppose that $(b^{(n)},\gamma^{(n)},v_{\mathrm{B}}^{(n)},v_{\mathrm{T}}^{(n)})\to(b,\gamma,v_{\mathrm{B}},v_{\mathrm{T}})$
in $\mathbb{R}^{2}\times\mathcal{X}_{\mathrm{BT}}$. It is clear that
  \[\mathscr{S}_{b^{(n)},\gamma^{(n)}}[v_{\mathrm{B}}^{(n)},v_{\mathrm{T}}^{(n)}]\to\mathscr{S}_{b,\gamma}[v_{\mathrm{B}},v_{\mathrm{T}}]\]
uniformly on compact subsets of $\mathbb{R}$. We note that, for each
$\ell>1/2$, we have
\[
\|\mathscr{S}_{b^{(n)},\gamma^{(n)}}[v_{\mathrm{B}}^{(n)},v_{\mathrm{T}}^{(n)}]\|_{\mathcal{C}_{\p_{\ell}}}\le\max\{\|v_{\mathrm{B}}^{(n)}\|_{\mathcal{C}_{\p_{\ell}}},\|v_{\mathrm{T}}^{(n)}\|_{\mathcal{C}_{\p_{\ell}}}\},
\]
so $(\mathscr{S}_{b^{(n)},\gamma^{(n)}}[v_{\mathrm{B}}^{(n)},v_{\mathrm{T}}^{(n)}])_{n}$
is bounded in each $\mathcal{C}_{\p_{\ell}}$, $\ell>1/2$. Therefore,
\[
\mathscr{S}_{b^{(n)},\gamma^{(n)}}[v_{\mathrm{B}}^{(n)},v_{\mathrm{T}}^{(n)}]\to\mathscr{S}_{b,\gamma}[v_{\mathrm{B}},v_{\mathrm{T}}]
\]
in the topology of each $\mathcal{C}_{\p_{\ell}}$, $\ell>1/2$, and
hence in the topology of $\mathcal{X}$.
\end{proof}

\section{Bottom and top solutions in the shock location reference frame\label{sec:BTsolns}}

In this section we consider what happens when we look at the bottom
and top solutions $u_{\mathrm{B}}$ and $u_{\mathrm{T}}$ in the reference
frame of the shock location $b_{t}$. We first compute the translation
formula
\[
\tau_{y}\mathscr{S}_{b,\gamma}[v_{\mathrm{B}},v_{\mathrm{T}}]=\mathscr{S}_{b+y,\gamma}[\tau_{y}(v_{\mathrm{B}},v_{\mathrm{T}})],
\]
which is easily checked from the definition \eqref{Sdef}. Therefore,
we can translate \eqref{uisS} in space to see that (with notation
as in that expression)
\begin{align*}
\tau_{b_{0}-b_{t}}(u_{\mathrm{B}},u_{\mathrm{T}},u^{[b_{0},\gamma]})(t,\cdot) & =(\tau_{b_{0}-b_{t}}(u_{\mathrm{B}},u_{\mathrm{T}})(t,\cdot),\mathscr{S}_{b_{0},\gamma}[\tau_{b_{0}-b_{t}}(u_{\mathrm{B}},u_{\mathrm{T}})(t,\cdot)]).
\end{align*}
We note that the right side depends only on $b_{0},\gamma$, and $\tau_{b_{0}-b_{t}}(u_{\mathrm{B}},u_{\mathrm{T}})(t,\cdot)$.
In other words, the shock~$u^{[b_{0},\gamma]}$ is a deterministic
and time-independent functional of the top and bottom solutions in
the reference frame of the shock location. Thus, in this section we
study just the translated top and botom solutions, i.e. $\tau_{b_{0}-b_{t}}(u_{\mathrm{B}},u_{\mathrm{T}})(t,\cdot)$.
The main results of this section concern the invariant measure in
this reference frame and its stability.

First we must define the evolution semigroup in the reference frame
of the shock. Given an initial condition $\mathbf{v}=(\mathbf{v}_{\mathrm{BT}},\tilde{\mathbf{v}})\in\mathcal{X}_{\mathrm{BT}}\times\mathcal{X}^{N}$,
with some $N\ge0$, let
\[
\mathbf{u}=(\mathbf{u}_{\mathrm{BT}},\tilde{\mathbf{u}})=\Psi(\mathbf{v})\in\mathcal{C}([0,\infty);\mathcal{X}_{\mathrm{BT}}\times\mathcal{X}^{N})
\]
solve \eqref{SBE-many} with $\mathbf{u}(0,\cdot)=\mathbf{v}$. As
in \cite{DGR19}, we define, for $F\in\mathcal{C}_{\mathrm{b}}(\mathcal{X}_{\mathrm{BT}}\times\mathcal{X}^{N})$
(a bounded continuous function on $\mathcal{X}_{\mathrm{BT}}\times\mathcal{X}^{N}$),
\[
P_{t}F(\mathbf{v})=\mathbb{E}F(\mathbf{u}(t,\cdot)),
\]
so that $\{P_{t}\}_{t\ge0}$ is the Markov semigroup for the dynamics
\eqref{SBE-many} in the original reference frame. Next let $\{b_{t}\}_{t\ge0}$
solve \eqref{bteqn} with initial condition $b_{0}=b$, set 
\[
\Phi^{[b]}(\mathbf{v})(t,\cdot)=\tau_{b-b_{t}}\mathbf{u}(t,\cdot)\in\mathcal{C}([0,\infty);\mathcal{X}_{\mathrm{BT}}\times\mathcal{X}^{N}),
\]
and put, again for $F\in\mathcal{C}_{\mathrm{b}}(\mathcal{X}_{\mathrm{BT}}\times\mathcal{X}^{N})$,
\[
\hat{P}_{t}^{[b]}F(\mathbf{v})=\mathbb{E}F(\Phi^{[b]}(\mathbf{v})(t,\cdot)).
\]
It is easily checked that $\{\hat{P}_{t}^{[b]}\}_{t\ge0}$ has the
semigroup property. It is the evolution semigroup in the reference
frame of the shock. Moreover, $\{\hat{P}_{t}^{[b]}\}_{t\ge0}$ has
the Feller property (which was checked for~$\{P_{t}\}_{t\ge0}$ in
\cite[Theorem 1.1]{DGR19}). We endow the space $\mathcal{C}([0,\infty);\mathcal{X}_{\mathrm{BT}}\times\mathcal{X}^{N})$
with the topology of uniform convergence (in the $\mathcal{X}_{\mathrm{BT}}\times\mathcal{X}^{N}$
norm) on compact subsets of $[0,\infty)$.
\begin{prop}
\label{prop:Phatfeller}The map $\Phi^{[b]}:\mathcal{X}_{\mathrm{BT}}\times\mathcal{X}\to\mathcal{C}([0,\infty);\mathcal{X}_{\mathrm{BT}}\times\mathcal{X}^{N})$
is continuous (with respect to the just-defined topology on the target).
Moreover, the semigroup
$\{\hat{P}_{t}^{[b]}\}_{t\ge0}$ has the Feller property: if $F\in\mathcal{C}_{\mathrm{b}}(\mathcal{X}_{\mathrm{BT}}\times\mathcal{X}^{N})$,
then~$\hat{P}_{t}^{[b]}F\in\mathcal{C}_{\mathrm{b}}(\mathcal{X}_{\mathrm{BT}}\times\mathcal{X}^{N})$
as well.
\end{prop}
\noindent We will prove \propref{Phatfeller} at the end of this section.

The first main result of this section concerns the invariance of the
tilted measures introduced in the statement of \thmref{stationary-shocks}.
\begin{prop}
\label{prop:preservetildenu}Let $\nu$ and $\hat{\nu}^{[b]}$ be
as in the statement of \thmref{stationary-shocks}. Then 
\begin{equation}
\hat{\nu}^{[b]}(\mathcal{X}_{\mathrm{BT}})=1\label{eq:nuhatXBtis1}
\end{equation}
 and 
\begin{equation}
(\hat{P}_{t}^{[b]})^{*}\hat{\nu}^{[b]}=\hat{\nu}^{[b]}.\label{eq:nuhatinvariant}
\end{equation}
\end{prop}

The second main result concerns the stability of the tilted measures
$\hat{\nu}_{a_{\mathrm{B}},a_{\mathrm{T}}}^{[b]}$ defined after the
statement of \thmref{stationary-shocks}.
\begin{prop}
\label{prop:convergetotildenu}Let $a_{\mathrm{B}}<a_{\mathrm{T}}$.
Let $\delta_{a_{\mathrm{B}},a_{\mathrm{T}}}$ be the measure on $\mathcal{X}_{\mathrm{BT}}$
with a single atom at the constant function $(a_{\mathrm{B}},a_{\mathrm{T}})$.
Then for any $b\in\mathbb{R}$, we have
\[
\lim_{t\to\infty}(\hat{P}_{t}^{[b]})^{*}\delta_{a_{\mathrm{B}},a_{\mathrm{T}}}=\hat{\nu}_{a_{\mathrm{B}},a_{\mathrm{T}}}^{[b]}
\]
weakly with respect to the topology of $\mathcal{X}_{\mathrm{BT}}$.
\end{prop}

The key ingredient in the proofs of \propref{preservetildenu}~and~\ref{prop:convergetotildenu}
is \propref{tiltedpreserved} below, which describes how a translation-invariant
measure evolves under $\hat{P}_{t}^{[b]}$. This will allow us to
tilt the invariant measures constructed in \cite{DGR19} to obtain
invariant measures in the reference frame of the shocks. We use the
notation from \cite{DGR19} that $\mathscr{P}_{\mathbb{R}}(\mathcal{X}_{\mathrm{BT}}\times\mathcal{X}^{N})$
is the space of translation-invariant probability measures on $\mathcal{X}_{\mathrm{BT}}\times\mathcal{X}^{N}$. (The subscript $\mathbb{R}$ denotes invariance under the action of $\mathbb{R}$ on the line by translations.)
If $\mu\in\mathscr{P}_{\mathbb{R}}(\mathcal{X}_{\mathrm{BT}})$ and
$(w_{\mathrm{B}},w_{\mathrm{T}})\sim\mu$, then (as noted in the statement
of \thmref{stationary-shocks}) the quantity
\[
B[w_{\mathrm{B}},w_{\mathrm{T}}]\coloneqq\lim_{L\to\infty}\frac{1}{L}\int_{0}^{L}[w_{\mathrm{T}}-w_{\mathrm{B}}](x)\,\dif x
\]
exists almost surely by the Birkhoff--Khinchin theorem.
\begin{prop}
\label{prop:tiltedpreserved}Let $N\ge0$. Let $\mu_{0}\in\mathscr{P}_\mathbb{R}(\mathcal{X}_\mathrm{BT}\times\mathcal{X}^N)$.
For each $t\ge0$, define another measure $\hat{\mu}_{t}^{[b]}$ on
$\mathcal{X}_{\mathrm{BT}}\times\mathcal{X}^{N}$, absolutely continuous
with respect to $\mu_{t}\coloneqq P_{t}^{*}\mu_{0}$, by
\begin{equation}
\frac{\dif\hat{\mu}_{t}^{[b]}}{\dif\mu_{t}}((w_{\mathrm{B}},w_{\mathrm{T}}),\tilde{\mathbf{w}})=\frac{w_{\mathrm{T}}(b)-w_{\mathrm{B}}(b)}{B[w_{\mathrm{B}},w_{\mathrm{T}}]}.\label{eq:RNderiv}
\end{equation}
Then $\hat{\mu}_{t}^{[b]}$ is a probability measure and
\begin{equation}
(\hat{P}_{t}^{[b]})^{*}\hat{\mu}_{0}^{[b]}=\hat{\mu}_{t}^{[b]}.\label{eq:muhatevolution}
\end{equation}
Moreover, for any $t\ge0$, if $\hat{\mathbf{v}}\sim\hat{\mu}_{t}^{[b]}$,
then for any deterministic $\zeta\in\mathbb{R}$, we have
\begin{equation}
\tau_{b-\overline{Z}_{b}[\hat{\mathbf{v}}]^{-1}(\zeta)}\hat{\mathbf{v}}\overset{\mathrm{law}}{=}\hat{\mathbf{v}}.\label{eq:shiftinvariant}
\end{equation}
\end{prop}

\begin{proof}
First we check that $\hat{\mu}_{t}^{[b]}$ is a probability measure.
Let $\mathcal{I}$ be the translation-invariant sub-$\sigma$-algebra
of the Borel $\sigma$-algebra on $\mathcal{X}_{\mathrm{BT}}\times\mathcal{X}^{N}$.
Then by the Birkhoff-Khinchin theorem, $B[w_{\mathrm{B}},w_{\mathrm{T}}]$
is~$\mathcal{I}$-measurable and in fact
\[
B[w_{\mathrm{B}},w_{\mathrm{T}}]=\mathbb{E}[w_{\mathrm{T}}(b)-w_{\mathrm{B}}(b)\mid\mathcal{I}]>0.
\]
It follows that
\begin{align*}
\mathbb{E}\left[\frac{w_{\mathrm{T}}(b)-w_{\mathrm{B}}(b)}{B[w_{\mathrm{B}},w_{\mathrm{T}}]}\right] & =\mathbb{E}\left[\mathbb{E}\left[\frac{w_{\mathrm{T}}(b)-w_{\mathrm{B}}(b)}{B[w_{\mathrm{B}},w_{\mathrm{T}}]}\ \middle|\ \mathcal{I}\right]\right]=\mathbb{E}\left[\frac{\mathbb{E}[w_{\mathrm{T}}(b)-w_{\mathrm{B}}(b)\mid\mathcal{I}]}{B[w_{\mathrm{B}},w_{\mathrm{T}}]}\right]=1,
\end{align*}
so $\hat{\mu}_{t}^{[b]}$ is a probability measure as claimed.

Let $\hat{\mathbf{v}}(t,\cdot)\sim\hat{\mu}_{t}^{[b]}$ for all $t\ge0$;
we will not use any coupling between $\hat{\mathbf{v}}(t,\cdot)$
and $\hat{\mathbf{v}}(s,\cdot)$ for $t\ne s$. Consider a function
$F\in L^{\infty}(\mathcal{X}_{\mathrm{BT}}\times\mathcal{X}^{N})$.
To prove \eqref{muhatevolution}, we need to show that
\begin{equation}
\mathbb{E}F(\hat{\mathbf{v}}(t,\cdot))=\mathbb{E}\hat{P}_{t}^{[b]}F(\hat{\mathbf{v}}^{[b]}(0,\cdot)).\label{eq:muhatevolutiongoalintermsofvhat}
\end{equation}
Let $\mathbf{u}\in\mathcal{C}([0,\infty);\mathcal{X}_{\mathrm{BT}}\times\mathcal{X}^{N}$)
solve \eqref{SBE-many} with initial condition $\mathbf{u}(0,\cdot)\sim\mu_{0}$
(independent of the noise). We abbreviate $Z_{b,t}=Z_{b,t}[\mathbf{u}(0,\cdot)]$.
We will show that both the left and right sides of \eqref{muhatevolutiongoalintermsofvhat}
are equal to
\[
\lim_{M\to\infty}\frac{1}{M}\int_{0}^{M}\mathbb{E}[F(\tau_{b-Z_{b,t}^{-1}(\zeta)}\mathbf{u}(t,\cdot)]\,\dif\zeta.
\]

We first show that
\begin{equation}
\mathbb{E}F(\hat{\mathbf{v}}(t,\cdot))=\lim_{M\to\infty}\frac{1}{M}\int_{0}^{M}\mathbb{E}[F(\tau_{b-Z_{b,t}^{-1}(\zeta)}\mathbf{u}(t,\cdot))]\,\dif\zeta.\label{eq:Evhatt}
\end{equation}
The crux of the argument is the simple identity
\begin{equation}
\frac{1}{L}\int_{0}^{L}F(\tau_{-x}\mathbf{u}(t,\cdot))[u_{\mathrm{T}}-u_{\mathrm{B}}](t,b+x)\,\dif x=\frac{2}{L}\int_{Z_{b,t}(b)}^{Z_{b,t}(L+b)}F(\tau_{b-Z_{b,t}^{-1}(\zeta)}\mathbf{u}(t,\cdot))\,\dif\zeta,\label{eq:chgvarbket}
\end{equation}
which comes from making the change of variables
\begin{equation}
x=Z_{b,t}^{-1}(\zeta)-b,\qquad\dif x=\frac{2}{[u_{\mathrm{T}}-u_{\mathrm{B}}](t,b+x)}\dif\zeta.\label{eq:chgvarintheintegral}
\end{equation}
By the Birkhoff--Khinchin theorem, we have the limit
\begin{equation}
\lim_{L\to\infty}\frac{1}{L}\int_{0}^{L}F(\tau_{-x}\mathbf{u}(t,\cdot))[u_{\mathrm{T}}-u_{\mathrm{B}}](t,b+x)\,\dif x=\mathbb{E}[F(\mathbf{u}(t,\cdot))[u_{\mathrm{T}}-u_{\mathrm{B}}](t,b)\mid\mathcal{I}]\label{eq:averagebyx}
\end{equation}
almost surely. Also by the Birkhoff-Khinchin theorem (recalling \eqref{Zaddition}),
we have
\begin{equation}
\lim_{L\to\infty}\frac{Z_{b,t}(L+b)-Z_{b,t}(b)}{L}=\frac{1}{2}B[(u_{\mathrm{B}},u_{\mathrm{T}})(t,\cdot)]\label{eq:ZstoB}
\end{equation}
almost surely. Combining \eqref{chgvarbket}--\eqref{ZstoB}, we
have
\begin{align*}
\lim_{L\to\infty}\frac{1}{Z_{b,t}(L+b)-Z_{b,t}(b)}\int_{Z_{b,t}(b)}^{Z_{b,t}(L+b)}F(\tau_{b-Z_{b,t}^{-1}(\zeta)}\mathbf{u}(t,\cdot))\,\dif\zeta & =\frac{\mathbb{E}[F(\mathbf{u}(t,\cdot))[u_{\mathrm{T}}-u_{\mathrm{B}}](t,b)\mid\mathcal{I}]}{B[(u_{\mathrm{B}},u_{\mathrm{T}})(t,\cdot)]}
\end{align*}
almost surely. Since $\lim\limits _{L\to\infty}Z_{b,t}(L+b)=\infty$
almost surely by \eqref{Zaddition} and \eqref{utinXBT}, and $F$
is bounded, this means that
\begin{equation}
\lim_{M\to\infty}\frac{1}{M}\int_{0}^{M}F(\tau_{b-Z_{b,t}^{-1}(\zeta)}\mathbf{u}(t,\cdot))\,\dif\zeta=\frac{\mathbb{E}[F(\mathbf{u}(t,\cdot))[u_{\mathrm{T}}-u_{\mathrm{B}}](t,b)\mid\mathcal{I}]}{B[(u_{\mathrm{B}},u_{\mathrm{T}})(t,\cdot)]}\label{eq:changeZtoM}
\end{equation}
almost surely. Since $F$ is bounded and $B[(u_{\mathrm{B}},u_{\mathrm{T}})(t,\cdot)]$
is $\mathcal{I}$-measurable, taking the expectation in~\eqref{changeZtoM}
and using the bounded convergence theorem we deduce that
\[
\lim_{M\to\infty}\frac{1}{M}\int_{0}^{M}\mathbb{E}[F(\tau_{b-Z_{b,t}^{-1}(\zeta)}\mathbf{u}(t,\cdot))]\,\dif\zeta=\mathbb{E}\left[\frac{F(\mathbf{u}(t,\cdot))[u_{\mathrm{T}}-u_{\mathrm{B}}](t,b)}{B[(u_{\mathrm{B}},u_{\mathrm{T}})(t,\cdot)]}\right],
\]
which implies \eqref{Evhatt} because $\mathbf{u}(t,\cdot)\sim\mu_{t}$.

The next step is to show that, for any $s\ge0$, we have
\begin{equation}
\mathbb{E}\hat{P}_{s}^{[b]}F(\hat{\mathbf{v}}^{[b]}(0,\cdot))=\lim_{M\to\infty}\frac{1}{M}\int_{0}^{M}\mathbb{E}[F(\tau_{b-Z_{b,s}^{-1}(\zeta)}\mathbf{u}(s,\cdot))]\,\dif\zeta.\label{eq:attimes}
\end{equation}
For a random variable $y\in\mathbb{R}$, measurable with respect to
$\mathbf{u}(0,\cdot)$, let $\mathbf{u}^{[y]}$ solve \eqref{SBE-many}
with initial condition $\mathbf{u}^{[y]}(0,\cdot)=\tau_{y}\mathbf{u}(0,\cdot)$.
We can compute
\begin{align}
\hat{P}_{s}^{[b]}F(\tau_{y}\mathbf{u}(0,\cdot)) & =\mathbb{E}[F(\tau_{b-Z_{b,s}[\tau_{y}\mathbf{u}(0,\cdot)]^{-1}(0)}\mathbf{u}^{[y]}(s,\cdot))\mid\mathbf{u}(0,\cdot)]\nonumber \\
 & =\mathbb{E}[F(\tau_{b-(Z_{b,s}[\tau_{y}\mathbf{u}(0,\cdot)]^{-1}(0)-y)}\tau_{-y}\mathbf{u}^{[y]}(s,\cdot))\mid\mathbf{u}(0,\cdot)]\nonumber \\
 & =\mathbb{E}[F(\tau_{b-Z_{b-y,s}^{-1}(0)}\mathbf{u}(s,\cdot))\mid\mathbf{u}(0,\cdot)].\label{eq:PtbsFrewrite}
\end{align}
The first equality above is by the definition of $\hat{P}_{s}^{[b]}$
and the second is a tautology. The third holds because by the translation-invariance
of the noise, $(Z_{b,s}[\tau_{y}\mathbf{u}(0,\cdot)]^{-1}(0)-y,\tau_{y}\mathbf{u}^{[y]}(s,\cdot))$
and $(Z_{b-y,s}^{-1}(0),\mathbf{u}(s,\cdot))$ have the same conditional
law given $\mathbf{u}(0,\cdot)$.

Now apply \eqref{Evhatt} with $t=0$ and $F$ in that equation taken
to be $\hat{P}_{s}^{[b]}F$. This gives
\begin{align}
\mathbb{E}[\hat{P}_{s}^{[b]}F(\hat{\mathbf{v}}(0,\cdot))] & =\lim_{M\to\infty}\frac{1}{M}\int_{0}^{M}\mathbb{E}\hat{P}_{s}^{[b]}F(\tau_{b-Z_{b,0}^{-1}(\zeta)}\mathbf{u}(0,\cdot))\,\dif\zeta\nonumber \\
 & =\lim_{M\to\infty}\frac{1}{M}\int_{0}^{M}\mathbb{E}F(\tau_{b-Z_{Z_{b,0}^{-1}(\zeta),s}^{-1}(0)}\mathbf{u}(s,\cdot))\,\dif\zeta,\label{eq:applytoPhat}
\end{align}
where in the second equality we used \eqref{PtbsFrewrite} with $y=b-Z_{b,0}^{-1}(\zeta)$.
Now we can compute
\begin{equation}
\begin{aligned}Z_{Z_{b,0}^{-1}(\zeta),s}(x) & =\frac{1}{2}\left(h_{\mathrm{T}}^{[Z_{b,0}^{-1}(\zeta)]}-h_{\mathrm{B}}^{[Z_{b,0}^{-1}(\zeta)]}\right)(s,x)=\frac{1}{2}\left(h_{\mathrm{T}}^{[b]}-h_{\mathrm{B}}^{[b]}\right)(s,x)-\frac{1}{2}\int_{b}^{Z_{b,0}^{-1}(\zeta)}[u_{\mathrm{T}}-u_{\mathrm{B}}](0,y)\,\dif y\\
 & =Z_{b,s}(x)-\zeta=Z_{b,s}(b)+\overline{Z}_{b}[\mathbf{u}(s,\cdot)](x)-\zeta.
\end{aligned}
\label{eq:rewriteZZx}
\end{equation}
In the second equality we used the fact that the identity
\[
\frac{1}{2}\left(h_{\mathrm{T}}^{[Z_{b,0}^{-1}(\zeta)]}-h_{\mathrm{B}}^{[Z_{b,0}^{-1}(\zeta)]}\right)(s,x)=\frac{1}{2}\left(h_{\mathrm{T}}^{[b]}-h_{\mathrm{B}}^{[b]}\right)(s,x)-\frac{1}{2}\int_{b}^{Z_{b,0}^{-1}(\zeta)}[u_{\mathrm{T}}-u_{\mathrm{B}}](0,y)\,\dif y
\]
for all $x\in\mathbb{R}$ holds at $s=0$ and thus for all $s\ge0$
as well. In the third equality of \eqref{rewriteZZx} we used~\eqref{Zaddition}.
It follows from \eqref{rewriteZZx} that
\begin{equation}
Z_{Z_{b,0}^{-1}(\zeta),s}^{-1}(\kappa)=\overline{Z}_{b}[\mathbf{u}(s,\cdot)]^{-1}(\kappa+\zeta-Z_{b,s}(b)).\label{eq:ZZinv}
\end{equation}
Substituting \eqref{ZZinv} (with $\kappa=0$) into \eqref{applytoPhat},
we get
\begin{align*}
\mathbb{E}[\hat{P}_{s}^{[b]}F(\hat{\mathbf{v}}(0,\cdot))] & =\lim_{M\to\infty}\frac{1}{M}\int_{0}^{M}\mathbb{E}F(\tau_{b-\overline{Z}_{b}[\mathbf{u}(s,\cdot)](x)(\zeta-Z_{b,s}(b))}\mathbf{u}(s,\cdot))\,\dif\zeta\\
 & =\lim_{M\to\infty}\frac{1}{M}\int_{-Z_{b,s}(b)}^{M-Z_{b,s}(b)}\mathbb{E}F(\tau_{b-\overline{Z}_{b}[\mathbf{u}(s,\cdot)]^{-1}(\zeta)}\mathbf{u}(s,\cdot))\\
 & =\mathbb{E}F(\hat{\mathbf{v}}(s,\cdot)),
\end{align*}
where the last equality is again by \eqref{Evhatt}, this time with
$t=s$. This completes the proof of \eqref{attimes}. As indicated
above, \eqref{Evhatt} and \eqref{attimes} together imply \eqref{muhatevolutiongoalintermsofvhat}.

The proof of \eqref{shiftinvariant} is similar but easier. Without
loss of generality, assume that $t=0$. Let $\mathbf{w}=(w_{\mathrm{B}},w_{\mathrm{T}},\tilde{\mathbf{w}})\sim\mu_{0}$
and $\hat{\mathbf{w}}\sim\hat{\mu}_{0}^{[b]}$. For $\mathbf{y}\in\mathcal{X}_{\mathrm{BT}}\times\mathcal{X}^{N}$
put $F_{b,\zeta}(\mathbf{y})=F(\tau_{b-\overline{Z}_{b}[\mathbf{y}]^{-1}(\zeta)}\mathbf{y})$.
By \eqref{Evhatt}, applied at $t=0$, we have
\begin{align}
\mathbb{E}F(\tau_{b-\overline{Z}_{b}[\hat{\mathbf{v}}]^{-1}(\zeta)}\hat{\mathbf{w}}) & =\mathbb{E}F_{b,\zeta}(\hat{\mathbf{w}})\nonumber \\
 & =\lim_{M\to\infty}\frac{1}{M}\int_{0}^{M}\mathbb{E}[F_{b,\zeta}(\tau_{b-\overline{Z}_{b}[\mathbf{w}]^{-1}(\zeta')}\mathbf{w})]\,\dif\zeta'\nonumber \\
 & =\lim_{M\to\infty}\frac{1}{M}\int_{0}^{M}\mathbb{E}[F(\tau_{b-\overline{Z}_{b}[\tau_{b-\overline{Z}_{b}[\mathbf{w}]^{-1}(\zeta')}\mathbf{w}]^{-1}(\zeta)}\tau_{b-\overline{Z}_{b}[\mathbf{w}]^{-1}(\zeta')}\mathbf{w})]\,\dif\zeta'.\label{eq:firstmanipfortransinvariance}
\end{align}
For any $y\in\mathbb{R}$, we have
\[
\overline{Z}_{b}[\tau_{y}\mathbf{w}]^{-1}(\zeta)=\overline{Z}_{b-y}[\mathbf{w}]^{-1}(\zeta)+y
\]
almost surely, which means that (taking $y=b-\overline{Z}_{b}^{-1}[\mathbf{w}](\zeta')$)
\[
\overline{Z}_{b}[\tau_{b-\overline{Z}_{b}^{-1}[\mathbf{w}](\zeta')}\mathbf{w}]^{-1}(\zeta)=\overline{Z}_{\overline{Z}_{b}^{-1}[\mathbf{w}](\zeta')}[\mathbf{w}]^{-1}(\zeta)+b-\overline{Z}_{b}^{-1}[\mathbf{w}](\zeta'),
\]
so
\begin{align*}
b-\overline{Z}_{b}[\tau_{b-\overline{Z}_{b}[\mathbf{w}]^{-1}(\zeta')}\mathbf{w}]^{-1}(\zeta) & =\overline{Z}_{b}^{-1}[\mathbf{w}](\zeta')-\overline{Z}_{\overline{Z}_{b}^{-1}[\mathbf{w}](\zeta')}[\mathbf{w}]^{-1}(\zeta)=\overline{Z}_{b}^{-1}[\mathbf{w}](\zeta')-\overline{Z}_{b}[\mathbf{w}]^{-1}(\zeta+\zeta'),
\end{align*}
where the second equality can be seen either by \eqref{ZZinv} with
$s=0$, $\zeta=\zeta'$, $\mathbf{u}(0,\cdot)=\mathbf{w}$, and $\kappa=\zeta$
or by a simple direct argument. Substituting this into \eqref{firstmanipfortransinvariance}
we obtain
\begin{align*}
\mathbb{E}F(\tau_{b-\overline{Z}_{b}[\hat{\mathbf{w}}]^{-1}(\zeta)}\hat{\mathbf{w}}) & =\lim_{M\to\infty}\frac{1}{M}\int_{0}^{M}\mathbb{E}[F(\tau_{b-\overline{Z}_{b}[\mathbf{w}]^{-1}(\zeta+\zeta')}\mathbf{w})]\,\dif\zeta'\\
 & =\lim_{M\to\infty}\frac{1}{M}\int_{\zeta}^{M+\zeta}\mathbb{E}[F(\tau_{b-\overline{Z}_{b}[\mathbf{w}]^{-1}(\zeta')}\mathbf{w})]\,\dif\zeta'=\mathbb{E}F(\hat{\mathbf{w}}),
\end{align*}
with the last equality again by \eqref{Evhatt}.
\end{proof}
Now we can prove \propref[s]{preservetildenu}~and~\ref{prop:convergetotildenu}.
\begin{proof}[Proof of \propref{preservetildenu}.]
The Birkhoff--Khinchin theorem, along with the assumed ordering of the components
of a function distributed according to $\nu$, implies that $\nu(\mathcal{X}_{\mathrm{BT}})=1$.
Then, since $\hat{\nu}^{[b]}$ is absolutely continuous with respect
to $\nu$, we have \eqref{nuhatXBtis1}. Then \eqref{nuhatinvariant}
holds by \propref{tiltedpreserved} applied with $\mu_{t}=\nu$ for
all $t$, since $\nu$ is an invariant measure for \eqref{SBE-many}.
\end{proof}
\begin{proof}[Proof of \propref{convergetotildenu}.]
We note that $(\hat{P}_{t}^{[b]})^{*}\delta_{a_{\mathrm{B}},a_{\mathrm{T}}}$
is ergodic with respect to the group $\mathbb{R}$ of spatial translations
due to the spatial ergodicity of the driving noise $V$. Therefore,
by \propref{tiltedpreserved} and the Birkhoff--Khinchin theorem,
$(\hat{P}_{t}^{[b]})^{*}\delta_{a_{\mathrm{B}},a_{\mathrm{T}}}$ is
absolutely continuous with respect to $P_{t}^{*}\delta_{a_{\mathrm{B}},a_{\mathrm{T}}}$
with Radon--Nikodym derivative
\[
\frac{\dif\left((\hat{P}_{t}^{[b]})^{*}\delta_{a_{\mathrm{B}},a_{\mathrm{T}}}\right)}{\dif(P_{t}^{*}\delta_{a_{\mathrm{B}},a_{\mathrm{T}}})}(w_{\mathrm{B}},w_{\mathrm{T}})=\frac{w_{\mathrm{T}}(b)-w_{\mathrm{B}}(b)}{a_{\mathrm{T}}-a_{\mathrm{B}}}.
\]
For any $F\in L^{\infty}(\mathcal{X}^{2})$, if $\hat{\mathbf{u}}^{[b]}(t,\cdot)\sim(\hat{P}_{t}^{[b]})^{*}\delta_{a_{\mathrm{B}},a_{\mathrm{T}}}$
and $\mathbf{u}(t,\cdot)=(u_{\mathrm{B}},u_{\mathrm{T}})(t,\cdot)\sim P_{t}^{*}\delta_{a_{\mathrm{B}},a_{\mathrm{T}}}$,
then we have by the definitions that
\begin{equation}
\mathbb{E}F(\hat{\mathbf{u}}^{[b]}(t,\cdot))=\mathbb{E}\left[F(\mathbf{u}(t,\cdot))\left(\frac{u_{\mathrm{T}}(t,b)-u_{\mathrm{B}}(t,b)}{a_{\mathrm{T}}-a_{\mathrm{B}}}\right)\right].\label{eq:EFuhat}
\end{equation}
By the $L^{2}$ bound proved as \cite[Lemma 5.3]{DGR19}, there is
a constant $C<\infty$ so that, for all $t\ge0$, we have
\[
\mathbb{E}\left(\frac{u_{\mathrm{T}}(t,b)-u_{\mathrm{B}}(t,b)}{a_{\mathrm{T}}-a_{\mathrm{B}}}\right)^{2}\le C.
\]
This means that the term inside the expectation on the right side
of \eqref{EFuhat} is uniformly integrable. Since $P_{t}^{*}\delta_{a_{\mathrm{B}},a_{\mathrm{T}}}$
converges to $\nu_{a_{\mathrm{B}},a_{\mathrm{T}}}$ (weakly with respect
to the topology of $\mathcal{X}^{2}$) by the stability result \cite[Theorem 1.3]{DGR19},
we have
  \begin{align*}
    \lim_{t\to\infty}F(\hat{\mathbf{u}}^{[b]}(t,\cdot))=\lim_{t\to\infty}\mathbb{E}\left[F(\mathbf{u}(t,\cdot))\left(\frac{u_{\mathrm{T}}(t,b)-u_{\mathrm{B}}(t,b)}{a_{\mathrm{T}}-a_{\mathrm{B}}}\right)\right]&=\mathbb{E}\left[F(\mathbf{v})\left(\frac{v_{\mathrm{T}}(b)-v_{\mathrm{B}}(b)}{a_{\mathrm{T}}-a_{\mathrm{B}}}\right)\right]\\&=\mathbb{E}F(\hat{\mathbf{v}}^{[b]}),
  \end{align*}
where $\mathbf{v}=(v_{\mathrm{B}},v_{\mathrm{T}})\sim\nu_{a_{\mathrm{B}},a_{\mathrm{T}}}$
and $\hat{\mathbf{v}}^{[b]}\sim\hat{\nu}_{a_{\mathrm{B}},a_{\mathrm{T}}}^{[b]}$.
Hence, $(\hat{P}_{t}^{[b]})^{*}\delta_{a_{\mathrm{B}},a_{\mathrm{T}}}$
converges weakly to $\hat{\nu}_{a_{\mathrm{B}},a_{\mathrm{T}}}^{[b]}$
with respect to the topology of $\mathcal{X}^{2}$.

It remains to show that in fact $(\hat{P}_{t}^{[b]})^{*}\delta_{a_{\mathrm{B}},a_{\mathrm{T}}}$
converges weakly to $\hat{\nu}_{a_{\mathrm{B}},a_{\mathrm{T}}}^{[b]}$
with respect to the topology of $\mathcal{X}_{\mathrm{BT}}$. If $F$
is a bounded Lipschitz function on $\mathcal{X}_{\mathrm{BT}}$, then
$F$ is in particular uniformly continuous, so it can be extended
to a bounded continuous function on the closure of $\mathcal{X}_{\mathrm{BT}}$
in $\mathcal{X}^{2}$, and hence by the Tietze extension theorem to
a bounded continuous function on $\mathcal{X}^{2}$. Then the argument
of the previous paragraph applies, and by the portmanteau lemma this
completes the proof.
\end{proof}

\subsection*{Proof of the Feller property}

Now we prove the Feller property \propref{Phatfeller}.
\begin{proof}[Proof of \propref{Phatfeller}.]
We recall that by \cite[Theorem 1.1]{DGR19}, the solution map $\Psi:\mathcal{X}_{\mathrm{BT}}\times\mathcal{X}^{N}\to\mathcal{C}([0,\infty);\mathcal{X}_{\mathrm{BT}}\times\mathcal{X}^{N})$
is continuous with probability $1$, if the target space is given
the topology of uniform convergence on compact subsets of $[0,\infty)$.
Thus, to show that $\Phi$ is continuous it suffices to show that
the map $\mathcal{X}_{\mathrm{BT}}\ni\mathbf{v}\mapsto(Z_{b,t}[\mathbf{v}]^{-1}(0))_{t\ge0}\in\mathcal{C}([0,\infty))$
is continuous with probability $1$, where $\mathcal{C}([0,\infty))$
is similarly given the topology of uniform convergence on compact
sets. This could be proved using the ODE \eqref{bteqn}, but we will
instead argue using the KPZ equation and the formulas of the previous
section.

Let $\mathbf{v}\in\mathcal{X}_{\mathrm{BT}}$ and let $\mathbf{u}\in\mathcal{C}([0,\infty);\mathcal{X}_{\mathrm{BT}})$
solve \eqref{SBE-many} with initial condition $\mathbf{u}(0,\cdot)=\mathbf{v}$.
Observe that, if $(h_{\mathrm{B}},h_{\mathrm{T}})$ solves \eqref{KPZ}
with initial condition
\[
(h_{\mathrm{B}},h_{\mathrm{T}})(0,x)=\int_{b}^{x}(v_{\mathrm{B}},v_{\mathrm{T}})(y)\,\dif y,
\]
so that
\begin{equation}
Z_{b,t}[\mathbf{v}](x)=\frac{1}{2}[h_{\mathrm{T}}-h_{\mathrm{B}}](t,x)=Z_{b,t}[\mathbf{v}](b)+\frac{1}{2}\int_{b}^{x}[u_{\mathrm{T}}-u_{\mathrm{B}}](t,y)\,\dif y,\label{eq:Zbtdecomp}
\end{equation}
then $Z_{b,t}[\mathbf{v}]$ satisfies the ODE
\begin{align}
\partial_{t}Z_{b,t}[\mathbf{v}](b) & =\frac{1}{2}\partial_{x}^{2}[h_{\mathrm{T}}-h_{\mathrm{B}}](t,b)-\frac{1}{2}\partial_{x}[h_{\mathrm{T}}-h_{\mathrm{B}}](t,b)\cdot\partial_{x}[h_{\mathrm{T}}+h_{\mathrm{B}}](t,b)\nonumber \\
 & =\frac{1}{2}\partial_{x}(u_{\mathrm{T}}-u_{\mathrm{B}})(t,b)-\frac{1}{2}(u_{\mathrm{T}}-u_{\mathrm{B}})(t,b)\cdot(u_{\mathrm{T}}+u_{\mathrm{B}})(t,b).\label{eq:dtZ}
\end{align}
Fix a smooth, compactly supported function $\varphi$ on $\mathbb{R}$
such that $\int\varphi=1$ and define
\begin{equation}
Q_{t}[\mathbf{v}]=\int_{\mathbb{R}}Z_{b,t}[\mathbf{v}](x)\varphi(x-b)\,\dif x.\label{eq:Qtdef}
\end{equation}
Integrating \eqref{dtZ} in time and against $\varphi(\cdot-b)$ in
space, and integrating by parts, we obtain
\begin{equation}
Q_{t}[\mathbf{v}]=\frac{1}{2}\int_{0}^{t}\int_{\mathbb{R}}(u_{\mathrm{T}}-u_{\mathrm{B}})(s,x)[-\varphi'(x-b)-(u_{\mathrm{T}}+u_{\mathrm{B}})(s,x)\varphi(x-b)]\,\dif x\,\dif s.\label{eq:QtIBP}
\end{equation}
On the other hand, using \eqref{Zbtdecomp} in \eqref{Qtdef} we can
also write
\begin{align}
Q_{t}[\mathbf{v}] & =\int_{\mathbb{R}}\left(Z_{b,t}[\mathbf{v}](b)+\frac{1}{2}\int_{b}^{x}[u_{\mathrm{T}}-u_{\mathrm{B}}](t,y)\,\dif y\right)\varphi(x-b)\,\dif x\nonumber \\
 & =Z_{b,t}[\mathbf{v}](b)+\frac{1}{2}\int_{\mathbb{R}}\int_{b}^{x}[u_{\mathrm{T}}-u_{\mathrm{B}}](t,y)\varphi(x-b)\,\dif y\,\dif x.\label{eq:Qtintermediate}
\end{align}
Combining \eqref{QtIBP} and \eqref{Qtintermediate} gives
\begin{align*}
Z_{b,t}[\mathbf{v}](b) & =\frac{1}{2}\int_{0}^{t}\int_{\mathbb{R}}(u_{\mathrm{T}}-u_{\mathrm{B}})(s,x)[-\varphi'(x-b)-(u_{\mathrm{T}}+u_{\mathrm{B}})(s,x)\varphi(x-b)]\,\dif x\,\dif s\\
 & \qquad-\frac{1}{2}\int_{\mathbb{R}}\int_{b}^{x}[u_{\mathrm{T}}-u_{\mathrm{B}}](t,y)\varphi(x-b)\,\dif y\,\dif x.
\end{align*}
By this and \cite[Theorem 1.1]{DGR19}, the map $\mathbf{v}\mapsto(Z_{b,t}[\mathbf{v}](b))_{t\ge0}$
is almost-surely continuous.

Now we can write, by \eqref{Zbtdecomp}, that
\[
Z_{b,t}[\mathbf{v}]^{-1}(0)=\overline{Z}_{b}[\mathbf{u}(t,\cdot)]^{-1}(-Z_{b,t}[\mathbf{v}](b)).
\]
By \lemref{inversionthing} and \cite[Theorem 1.1]{DGR19}, this implies
that $\mathbf{v}\mapsto(Z_{b,t}[\mathbf{v}]^{-1}(0))_{t\ge0}$ is
almost-surely continuous. Therefore, $\Phi$ is almost-surely continuous.
Now if $F\in\mathcal{C}_{\mathrm{b}}(\mathcal{X}_{\mathrm{BT}}\times\mathcal{X}^{N})$
and~$\mathbf{v}_{n}\to\mathbf{v}$ in~$\mathcal{X}_{\mathrm{BT}}\times\mathcal{X}^{N}$
then
\[
\hat{P}_{t}^{[b]}F(\mathbf{v}_{n})=\mathbb{E}F(\Phi(\mathbf{v}_{n})(t,\cdot))=\mathbb{E}F(\Phi(\mathbf{v})(t,\cdot))=\hat{P}_{t}^{[b]}F(\mathbf{v})
\]
by the bounded convergence theorem. This proves that $\hat{P}_{t}^{[b]}$
is Feller.
\end{proof}

\section{Uniqueness of the stationary shock profiles\label{sec:uniquenessofshockprofiles}}

In this section we show that shock profiles of the form \eqref{shockprototype}
are the only ``stationary'' shock profiles that satisfy a certain
integrability condition. We define this integrability condition through
the space
\begin{equation}
\mathcal{X}_{\mathrm{Sh}}=\left\{ (v_{\mathrm{B}},v_{\mathrm{T}},v)\in\mathcal{X}_{\mathrm{BT}}\times\mathcal{X}\ :\ \int_{-\infty}^{0}|v_{\mathrm{T}}-v|+\int_{0}^{\infty}|v-v_{\mathrm{B}}|<\infty\right\} ,\label{eq:XShdef}
\end{equation}
as previously given in \eqref{XShdefbis2}.
This is a space of viscous shock fronts. As in the previous sections,
$v_{\mathrm{B}}$ and $v_{\mathrm{T}}$ are the ``bottom'' and ``top''
solutions, respectively, while $v$ is a viscous shock. Note that
for any $(v_{\mathrm{B}},v_{\mathrm{T}})\in\mathcal{X}_{\mathrm{BT}}$
and $b,\gamma\in\mathbb{R}$, we have $\mathscr{S}_{b,\gamma}[v_{\mathrm{B}},v_{\mathrm{T}}]\in\mathcal{X}_{\mathrm{Sh}}$.

Next, we need a way to track the location of a moving shock. We define
\[
\mathcal{X}_{\mathrm{Sh},b,\gamma}=\left\{ (v_{\mathrm{B}},v_{\mathrm{T}},v)\in\mathcal{X}_{\mathrm{Sh}}\ :\ \int_{-\infty}^{b}[v-v_{\mathrm{T}}]+\int_{b}^{\infty}[v-v_{\mathrm{B}}]=\gamma\right\} .
\]
Observe that (recalling the definition \eqref{XBTdef}), for any $(v_{\mathrm{B}},v_{\mathrm{T}},v)\in\mathcal{X}_{\mathrm{Sh}}$,
the map
\[
I(c)=\int_{-\infty}^{c}[v-v_{\mathrm{T}}]+\int_{c}^{\infty}[v-v_{\mathrm{B}}]
\]
is decreasing and, moreover,
\[
\lim_{c\to\pm\infty}I(c)=\mp\infty.
\]
Therefore, for each fixed $b\in\mathbb{R}$, we have
\[
\mathcal{X}_{\mathrm{Sh}}=\bigsqcup_{\gamma\in\mathbb{R}}\mathcal{X}_{\mathrm{Sh},b,\gamma},
\]
and for each fixed $\gamma\in\mathbb{R}$, we have
\[
\mathcal{X}_{\mathrm{Sh}}=\bigsqcup_{b\in\mathbb{R}}\mathcal{X}_{\mathrm{Sh},b,\gamma},
\]
where $\bigsqcup$ denotes disjoint union.

We now show that the shocks \eqref{shockprototype} lie in the corresponding
$\mathcal{X}_{\mathrm{Sh},b,\gamma}$.
\begin{lem}
\label{lem:shocksinXshb}We have $(v_{\mathrm{B}},v_{\mathrm{T}},\mathscr{S}_{b,\gamma}[v_{\mathrm{B}},v_{\mathrm{T}}])\in\mathcal{X}_{\mathrm{Sh},b,\gamma}$
for any $(v_{\mathrm{B}},v_{\mathrm{T}})\in\mathcal{X}_{\mathrm{BT}}$
and any $b,\gamma\in\mathbb{R}$. 
\end{lem}

\begin{proof}
By \eqref{Sbgammarewrite} and the change of variables
\[
\zeta=\overline{Z}_{b}[v_{\mathrm{B}},v_{\mathrm{T}}](x),\qquad\dif\zeta=\frac{1}{2}[v_{\mathrm{T}}-v_{\mathrm{B}}](x)\dif x,
\]
(similar to \eqref{chgvar-intro}), we have
\begin{align}
\int_{-\infty}^{b}(\mathscr{S}_{b,\gamma}[v_{\mathrm{B}},v_{\mathrm{T}}]-v_{\mathrm{T}})(x)\,\dif x & =\int_{-\infty}^{b}\frac{v_{\mathrm{B}}(x)-v_{\mathrm{T}}(x)}{1+\e^{\gamma-2\overline{Z}_{b}[v_{\mathrm{B}},v_{\mathrm{T}}](x)}}\,\dif x\nonumber \\
 & =-2\int_{-\infty}^{0}\frac{1}{1+\e^{\gamma-2\zeta}}\,\dif\zeta=-\log(1+\e^{-\gamma}).\label{eq:leftshock}
\end{align}
Similarly, we have
\begin{equation}
\int_{b}^{\infty}(\mathscr{S}_{b,\gamma}[v_{\mathrm{B}},v_{\mathrm{T}}]-v_{\mathrm{B}})(x)\,\dif x=\int_{b}^{\infty}\frac{v_{\mathrm{T}}(x)-v_{\mathrm{B}}(x)}{1+\e^{2\overline{Z}_{b}[v_{\mathrm{B}},v_{\mathrm{T}}](x)-\gamma}}\,\dif x=2\int_{0}^{\infty}\frac{1}{1+\e^{2\zeta-\gamma}\,\dif\zeta}=\log(1+\e^{\gamma}).\label{eq:rightshock}
\end{equation}
Adding \eqref{leftshock} and \eqref{rightshock} yields
\[
\int_{-\infty}^{b}(\mathscr{S}_{b,\gamma}[v_{\mathrm{B}},v_{\mathrm{T}}]-v_{\mathrm{T}})(x)\,\dif x+\int_{b}^{\infty}(\mathscr{S}_{b,\gamma}[v_{\mathrm{B}},v_{\mathrm{T}}]-v_{\mathrm{B}})(x)\,\dif x=\gamma,
\]
completing the proof.
\end{proof}
The following simple lemma gives an alternative characterization of
$\mathcal{X}_{\mathrm{Sh},b,\gamma}$.
\begin{lem}
\label{lem:Xshequivalence}We have the equivalence
\begin{equation}
(v_{\mathrm{B}},v_{\mathrm{T}},v)\in\mathcal{X}_{\mathrm{Sh},b,\gamma}\iff v-\mathscr{S}_{b,\gamma}[v_{\mathrm{B}},v_{\mathrm{T}}]\in L^{1}(\mathbb{R})\text{ and }\int_{\mathbb{R}}(v-\mathscr{S}_{b,\gamma}[v_{\mathrm{B}},v_{\mathrm{T}}])=0.\label{eq:XShequivalence}
\end{equation}
\end{lem}

\begin{proof}
If $(v_{\mathrm{B}},v_{\mathrm{T}},v)\in\mathcal{X}_{\mathrm{Sh},b,\gamma}$
and $(v_{\mathrm{B}},v_{\mathrm{T}},\tilde{v})\in\mathcal{X}_{\mathrm{Sh},b,\tilde{\gamma}}$,
then
\begin{equation}
\int_{-\infty}^{\infty}[v-\tilde{v}](y)\,\dif y=\int_{-\infty}^{b}[v-v_{\mathrm{T}}]+\int_{-\infty}^{b}[v_{\mathrm{T}}-\tilde{v}]+\int_{b}^{\infty}[v-v_{\mathrm{B}}]+\int_{b}^{\infty}[v_{\mathrm{B}}-\tilde{v}]=\gamma-\tilde{\gamma}.\label{eq:differenceisdifferenceofgammas}
\end{equation}
Combining \eqref{differenceisdifferenceofgammas} and \lemref{shocksinXshb}
yields the ``$\implies$'' direction of \eqref{XShequivalence}.

On the other hand, if $v-\mathscr{S}_{b,\gamma}[v_{\mathrm{B}},v_{\mathrm{T}}]\in L^{1}(\mathbb{R})$,
then $(v_{\mathrm{B}},v_{\mathrm{T}},v)\in\mathcal{X}_{\mathrm{Sh}}$
since $\mathscr{S}_{b,\gamma}[v_{\mathrm{B}},v_{\mathrm{T}}]\in\mathcal{X}_{\mathrm{Sh}}$.
Thus the ``$\impliedby$'' direction of \eqref{XShequivalence}
follows immediately from the second equality in \eqref{differenceisdifferenceofgammas}.
\end{proof}
The next lemma shows that for an arbitrary shock, the shock location
follows the location $b_{t}$.
\begin{lem}
\label{lem:shockcentersmove}Suppose that $\mathbf{u}=(u_{\mathrm{B}},u_{\mathrm{T}},u)\in\mathcal{C}([0,\infty);\mathcal{X}_{\mathrm{BT}}\times\mathcal{X})$
is a solution to \eqref{SBE-many} such that $\mathbf{u}(0,\cdot)\in\mathcal{X}_{\mathrm{Sh},b,\gamma}$,
and let $(b_{t})_{t\ge0}$ solve with \eqref{bteqn} with initial
condition $b_{0}=b$. Then, with probability~$1$, for all $t\ge0$,
we have $\mathbf{u}(t,\cdot)\in\mathcal{X}_{\mathrm{Sh},b_{t},\gamma}$.
\end{lem}

\begin{proof}
Define
\[
u_{\mathrm{explicit}}(t,x)=\mathscr{S}_{b_{t},\gamma}[(u_{\mathrm{B}},u_{\mathrm{T}})(t,\cdot)](x),
\]
so $(u_{\mathrm{B}},u_{\mathrm{T}},u,u_{\mathrm{explicit}})$ solves
\eqref{SBE-many} by \eqref{uisS}. Now by the mass conservation of
the Burgers dynamics (\cite[Proposition 3.3]{DGR19}) we have
\begin{align*}
0=\int_{\mathbb{R}}(u_{\mathrm{expicit}}(0,x)-u(0,x))\,\dif x & =\int_{\mathbb{R}}(u_{\mathrm{explicit}}(t,x)-u(t,x))\,\dif x\\
 & =\int_{\mathbb{R}}(\mathscr{S}_{b_{t},\gamma}[(u_{\mathrm{B}},u_{\mathrm{T}})(t,\cdot)](x)-u(t,x))\,\dif x.
\end{align*}
Thus \lemref{Xshequivalence} implies that $\mathbf{u}(t,\cdot)\in\mathcal{X}_{\mathrm{Sh},b_{t},\gamma}$
for all $t\ge0$.
\end{proof}
\begin{defn}
\label{def:stationaryshockprofile}Let $\mu$ be a probability measure
on $\mathcal{X}_{\mathrm{Sh}}$ and $b\in\mathbb{R}$. We say that
$\mu$ is the law of a \emph{stationary shock profile} with respect
to $b$ if $(\hat{P}_{t}^{[b]})^{*}\mu=\mu$.
\end{defn}

An immediate consequence of \eqref{uisS} and \propref{preservetildenu}
is that if $a_{\mathrm{B}}<a_{\mathrm{T}}$, $b\in\mathbb{R}$, and
$(v_{\mathrm{B}},v_{\mathrm{T}})\sim\hat{\nu}_{a_{\mathrm{B}},a_{\mathrm{T}}}^{[b]}$
(as defined in \eqref{chgmeasure-intro}), then for any $\gamma\in\mathbb{R}$,
$(v_{\mathrm{B}},v_{\mathrm{T}},\mathscr{S}_{b,\gamma}[v_{\mathrm{B}},v_{\mathrm{T}}])$
has the law of a stationary shock profile with respect to $b$. We
can also prove a partial converse of this property.
\begin{prop}
\label{prop:classifyssps}If $\mathbf{v}=(v_{\mathrm{B}},v_{\mathrm{T}},v)$
has the law of a stationary shock profile with respect to $b$, then
there is a random $\gamma\in\mathbb{R}$ so that $v=\mathscr{S}_{b,\gamma}[v_{\mathrm{B}},v_{\mathrm{T}}]$
almost surely.
\end{prop}

\begin{proof}
Let $\gamma$ be such that $\mathbf{v}\in\mathcal{X}_{\mathrm{Sh},b,\gamma}$
and $\mathbf{u}=(u_{\mathrm{B}},u_{\mathrm{T}},u,u_{\mathrm{explicit}})$
solve \eqref{SBE-many} with initial condition
\[
\mathbf{u}(0,\cdot)=(\mathbf{v},\mathscr{S}_{b,\gamma}[v_{\mathrm{B}},v_{\mathrm{T}}]).
\]
Then, in particular, if $(b_{t})_{t\ge0}$ solves \eqref{bteqn} with
initial condition $b_{0}=b$, then
\[
u_{\mathrm{explicit}}(t,x)=\mathscr{S}_{b_{t},\gamma}[(u_{\mathrm{B}},u_{\mathrm{T}})(t,\cdot)](x)
\]
by \eqref{uisS}. By the mass conservation of the Burgers equation
(\cite[Proposition 3.3]{DGR19}) we also know that
\begin{equation}
\int_{\mathbb{R}}(u-u_{\mathrm{explicit}})(t,x)\,\dif x=0\label{eq:differenceis0}
\end{equation}
for all $t\ge0$. Since $(v_{\mathrm{B}},v_{\mathrm{T}},v)$ is a
stationary shock profile, it follows that
\[
0=\frac{\dif}{\dif t}\int_{\mathbb{R}}|\tau_{-b_{t}}u(t,\cdot)-\tau_{-b_{t}}u_{\mathrm{explicit}}(t,\cdot)|=\frac{\dif}{\dif t}\int_{\mathbb{R}}|u(t,\cdot)-u_{\mathrm{explicit}}(t,\cdot)|.
\]
This allows us to use the ordering result proved in \cite[Proposition 3.9]{DGR19}
(using hypothesis (H2') there) which then implies that $u$ and $u_{\mathrm{explicit}}$
must be ordered almost surely. In light of \eqref{differenceis0},
this means that $v=\mathscr{S}_{b,\gamma}[v_{\mathrm{B}},v_{\mathrm{T}}]$
almost surely, as claimed.
\end{proof}

\section{Stability of the viscous shocks\label{sec:stability}}

In this section we study the stability of the viscous shocks \eqref{shockprototype}
and prove \thmref{shock-stability}. The proof follows a strategy,
based on ordering and $L^{1}$ contraction, similar to \cite{DGR19}.
We begin with a time-averaged result.
\begin{prop}
\label{prop:timeaveragedconvergence}Fix real numbers $a_{\mathrm{B}}<a_{\mathrm{T}}$
and $\gamma_{\mathrm{L}}<\gamma_{\mathrm{R}}$. Let $(u_{\mathrm{B}},u_{\mathrm{T}},u)\in\mathcal{C}([0,\infty);\mathcal{X}_{\mathrm{BT}}\times\mathcal{X})$
solve \eqref{SBE-many} with initial conditions satisfying $u_{\mathrm{Y}}(0,\cdot)\equiv a_{\mathrm{Y}}$
for $\mathrm{Y}\in\{\mathrm{B},\mathrm{T}\}$. Let $b,\gamma\in\mathbb{R}$
be such that $(a_{\mathrm{B}},a_{\mathrm{T}},u(0,\cdot))\in\mathcal{X}_{\mathrm{Sh},b,\gamma}$,
and let $(b_{t})_{t\ge0}$ solve \eqref{bteqn} with initial condition
$b_{0}=b$. Further asume that for all $x\in\mathbb{R}$, we have
\begin{equation}
\mathscr{S}_{b,\gamma_{\mathrm{L}}}[a_{\mathrm{B}},a_{\mathrm{T}}](x)\le u(0,x)\le\mathscr{S}_{b,\gamma_{\mathrm{R}}}[a_{\mathrm{B}},a_{\mathrm{T}}](x).\label{eq:betweentwoshocks}
\end{equation}
If $(w_{\mathrm{B}},w_{\mathrm{T}})\sim\hat{\nu}_{a_{\mathrm{B}},a_{\mathrm{T}}}^{[b]}$,
then we have
\[
\lim_{T\to\infty}\int_{1}^{T+1}\Law(\tau_{b-b_{t}}(u_{\mathrm{B}},u_{\mathrm{T}},u)(t,\cdot))=\Law(w_{\mathrm{B}},w_{\mathrm{T}},\mathscr{S}_{b,\gamma}[w_{\mathrm{B}},w_{\mathrm{T}}])
\]
weakly with respect to the topology of $\mathcal{X}_{\mathrm{BT}}\times\mathcal{X}$.
\end{prop}

\begin{proof}
Let us define $\gamma_{\mathrm{C}}=\gamma$ for simplicity of notation
later on. Consider the joint families 
\[
\tilde{\mathbf{u}}=(u_{\mathrm{B}},u_{\mathrm{T}},u)\in\mathcal{C}([0,\infty);\mathcal{X}_{\mathrm{BT}}\times\mathcal{X})
\]
and
\[
\mathbf{u}=(\mathbf{u}_{\mathrm{explicit}},u)=((u_{\mathrm{B}},u_{\mathrm{T}},u_{\mathrm{L}},u_{\mathrm{C}},u_{\mathrm{R}}),u)\in\mathcal{C}([0,\infty);\mathcal{X}_{\mathrm{BT}}\times\mathcal{X}^{4})
\]
solving \eqref{SBE-many} with initial conditions
\[
u_{\mathrm{X}}(0,\cdot)=\mathscr{S}_{b,\gamma_{\mathrm{X}}}[v_{\mathrm{B}},v_{\mathrm{T}}]
\]
for $\mathrm{X}\in\{\mathrm{L},\mathrm{C},\mathrm{R}\}$. We note
that
\begin{equation}
\tau_{b-b_{t}}u_{\mathrm{X}}(t,\cdot)=\mathscr{S}_{b,\gamma_{\mathrm{X}}}[\tau_{b-b_{t}}(u_{\mathrm{B}},u_{\mathrm{T}})(t,\cdot)]\label{eq:preservestuff}
\end{equation}
for $\mathrm{X}\in\{\mathrm{L},\mathrm{C},\mathrm{R}\}$. Also, the
comparison principle and \eqref{betweentwoshocks} imply that
\[
u_{\mathrm{L}}(t,x)\le u(t,x)\le u_{\mathrm{R}}(t,x),\qquad\text{for all \ensuremath{t\ge0} and \ensuremath{x\in\mathbb{R}}}
\]
and
\[
u_{\mathrm{L}}(t,x)\le u_{\mathrm{C}}(t,x)\le u_{\mathrm{R}}(t,x),\qquad\text{for all \ensuremath{t\ge0} and \ensuremath{x\in\mathbb{R}}.}
\]
In addition, by \lemref[s]{Xshequivalence}~and~\ref{lem:shockcentersmove},
we have
\begin{equation}
\int_{\mathbb{R}}[u-u_{\mathrm{C}}](t,x)\,\dif x=0.\label{eq:intwithCiszero}
\end{equation}
Therefore, we have, for $\mathrm{X}\in\{\mathrm{L},\mathrm{R}\}$,
that
\[
\|\tau_{b-b_{t}}[u-u_{\mathrm{X}}](t,\cdot)\|_{L^{1}(\mathbb{R})}=\left|\int_{\mathbb{R}}[u-u_{\mathrm{X}}](t,x)\,\dif x\right|=\left|\int_{\mathbb{R}}[u_{\mathrm{C}}-u_{\mathrm{X}}](t,\cdot)\right|=|\gamma-\gamma_{\mathrm{X}}|,
\]
with the second equality by \eqref{intwithCiszero} and the third
by \propref{L1explicit}. 

We claim that the family $\{\tau_{b-b_{t}}\mathbf{u}(t,\cdot)\}_{t\ge1}$
is tight in $\mathcal{X}_{\mathrm{BT}}\times\mathcal{X}$. Indeed,
by \propref{convergetotildenu}, the family $\{\tau_{b-b_{t}}(u_{\mathrm{B}},u_{\mathrm{T}})(t,\cdot)\}_{t\ge0}$
converges in law with respect to the topology of $\mathcal{X}_{\mathrm{BT}}$
as $t\to\infty$, so in particular by Prokhorov's theorem (which applies
since $\mathcal{X}_{\mathrm{BT}}$ is a Polish space as proved in
\lemref{XBTPolish}) this family is tight in $\mathcal{X}_{\mathrm{BT}}$.
By the comparison principle (\cite[Proposition 3.1]{DGR19} we have
\[
u_{\mathrm{B}}(t,x)\le u(t,x)\le u_{\mathrm{T}}(t,x)
\]
for all $t\ge0$ and $x\in\mathbb{R}$, so the family $\{\tau_{b-b_{t}}u(t,\cdot)\}_{t\ge0}$
is uniformly bounded in probability in $\mathcal{X}$. Then \cite[Proposition 2.2]{DGR19}
this implies that $\{\tau_{b-b_{t}}u(t,\cdot)\}_{t\ge1}$ is tight
in $\mathcal{X}$. Therefore, $\{\tau_{b-b_{t}}\mathbf{u}(t,\cdot)\}_{t\ge1}$
is tight in the topology of $\mathcal{X}_{\mathrm{BT}}\times\mathcal{X}^{4}.$

Now let $T_{k}\uparrow\infty$ be a sequence so that
\[
\mu=\lim_{k\to\infty}\frac{1}{T_{k}}\int_{1}^{T_{k+1}}\Law(\tau_{b-b_{t}}\mathbf{u}(t,\cdot))\,\dif t
\]
exists in the sense of weak convergence of probability measures on
$\mathcal{X}_{\mathrm{BT}}$. Consider
\[
\mathbf{w}=(w_{\mathrm{B}},w_{\mathrm{T}},w_{\mathrm{L}},w_{\mathrm{C}},w_{\mathrm{R}})\sim\mu,
\]
and $\tilde{\mathbf{w}}=(w_{\mathrm{B}},w_{\mathrm{T}},w)$. By \eqref{preservestuff}
and \propref{Scts}, we have $w_{\mathrm{X}}=\mathscr{S}_{b,\gamma_{X}}[w_{\mathrm{B}},w_{\mathrm{T}}]$
for $\mathrm{X}\in\{\mathrm{L},\mathrm{C},\mathrm{R}\}$ almost surely.
By the Skorokhod representation theorem, Fatou's lemma, and the
$L^{1}(\mathbb{R})$ contraction property of the Burgers equation
as stated in \cite[Proposition 3.2]{DGR19}, we therefore have
\begin{equation}
\|w-w_{\mathrm{C}}\|_{L^{1}(\mathbb{R})}\le\|(u-u_{\mathrm{C}})(0,\cdot)\|_{L^{1}(\mathbb{R})}<\infty\label{eq:wminuswC}
\end{equation}
almost surely. Similarly, for $\mathrm{X}\in\{\mathrm{L},\mathrm{R}\}$,
we have
\begin{equation}
\|w-w_{\mathrm{X}}\|_{L^{1}(\mathbb{R})}\le\|(u-u_{\mathrm{X}})(0,\cdot)\|_{L^{1}(\mathbb{R})}=|\gamma_{\mathrm{X}}-\gamma|\label{eq:distancetwosides}
\end{equation}
almost surely. We see from \eqref{wminuswC} that $(w_{\mathrm{B}},w_{\mathrm{T}},w)\in\mathcal{X}_{\mathrm{Sh}}$
almost surely. Moreover, the Krylov--Bogolyubov theorem (see e.g.
\cite[Theorem 3.1.1]{DPZ96}) tells us that
\[
(\hat{P}_{t}^{[b]})^{*}\Law(\tilde{\mathbf{w}})=\Law(\tilde{\mathbf{w}})\qquad\text{for any \ensuremath{t\ge0}.}
\]
Therefore, $\tilde{\mathbf{w}}$ is a stationary shock profile in
the sense of \defref{stationaryshockprofile}. By \propref{classifyssps},
there is a random $\tilde{\gamma}\in\mathbb{R}$ so that $w\in\mathscr{S}_{b,\tilde{\gamma}}[w_{\mathrm{B}},w_{\mathrm{T}}]$
with probability $1$. This means that $\|w-w_{\mathrm{X}}\|_{L^{1}(\mathbb{R})}=|\gamma_{\mathrm{X}}-\tilde{\gamma}|$
for $\mathrm{X}\in\{\mathrm{L},\mathrm{R}\}$. Combined with \eqref{distancetwosides},
this means that $\tilde{\gamma}=\gamma$ almost surely. This uniquely
identifies~$\mu$. Since the topology of weak convergence
of probability measures is metrizable, we therefore have
\[
\lim_{T\to\infty}\frac{1}{T}\int_{1}^{T+1}\Law(\tau_{b-b_{t}}\tilde{\mathbf{u}}(t,\cdot))\,\dif t=\mu
\]
weakly with respect to the topology of $\mathcal{X}_{\mathrm{BT}}\times\mathcal{X}$,
as claimed.
\end{proof}
The next proposition shows the almost sure $L^{1}(\mathbb{R})$ convergence
of the solution to an initial value problem to a viscous shock arising
from a corresponding shift of $(u_{\mathrm{B}},u_{\mathrm{T}})$.
\begin{prop}
\label{prop:convergenceinL1}With the same notation and assumptions
as \propref{timeaveragedconvergence}, we have
\[
\lim_{t\to\infty}\|\tau_{b-b_{t}}u(t,\cdot)-\mathscr{S}_{b,\gamma}[\tau_{b-b_{t}}(u_{\mathrm{B}},u_{\mathrm{T}})(t,\cdot)]\|_{L^{1}(\mathbb{R})}=0
\]
almost surely.
\end{prop}

To prove \propref{convergenceinL1}, we first prove the following
lemma.
\begin{lem}
\label{lem:splitupL1}Suppose that $(v_{\mathrm{B}},v_{\mathrm{T}},v),(v_{\mathrm{B}},v_{\mathrm{T}},\tilde{v})\in\mathcal{X}_{\mathrm{Sh}}$
and, for some $b\in\mathbb{R}$ and $\gamma_{\mathrm{L}}<\gamma_{\mathrm{R}}$
we have
\begin{align}
\mathscr{S}_{b,\gamma_{\mathrm{L}}}[v_{\mathrm{B}},v_{\mathrm{T}}](x) & \le v(x)\le\mathscr{S}_{b,\gamma_{\mathrm{R}}}[v_{\mathrm{B}},v_{\mathrm{T}}](x),\label{eq:vbetween}\\
\mathscr{S}_{b,\gamma_{\mathrm{L}}}[v_{\mathrm{B}},v_{\mathrm{T}}](x) & \le\tilde{v}(x)\le\mathscr{S}_{b,\gamma_{\mathrm{R}}}[v_{\mathrm{B}},v_{\mathrm{T}}](x)\label{eq:vtildebetween}
\end{align}
for all $x\in\mathbb{R}$. Then there is a constant $C<\infty$, depending
only on $\gamma_{\mathrm{L}}$ and $\gamma_{\mathrm{R}}$, so that
for all $L>0$ and all $\ell>1/2$ we have
\begin{equation}
\|v-\tilde{v}\|_{L^{1}(\mathbb{R})}\le2\langle|b|+L\rangle^{1+\ell}\|v-\tilde{v}\|_{\mathcal{C}_{\p_{\ell}}}+C\left(\e^{2\overline{Z}_{b}[v_{\mathrm{B}},v_{\mathrm{T}}](b-L)}+\e^{-2\overline{Z}_{b}[v_{\mathrm{B}},v_{\mathrm{T}}](b+L)}\right).\label{eq:vvtildeL1}
\end{equation}
\end{lem}

\begin{proof}
For each $L>0$, we have
\begin{equation}
\|v-\tilde{v}\|_{L^{1}(\mathbb{R})}=\|v-\tilde{v}\|_{L^{1}([b-L,b+L])}+\|v-\tilde{v}\|_{L^{1}(\mathbb{R}\setminus[b-L,b+L])},\label{eq:splitupL1}
\end{equation}
and
\begin{equation}
\|v-\tilde{v}\|_{L^{1}([b-L,b+L])}\le2\langle|b|+L\rangle^{1+\ell}\|v-\tilde{v}\|_{\mathcal{C}_{\p_{\ell}}}.\label{eq:centerpart}
\end{equation}
Using \eqref{vbetween}--\eqref{vtildebetween} and arguing as in
\propref{L1explicit}, we have 
\begin{align}
\int_{-\infty}^{b-L}|[v-\tilde{v}](x)|\,\dif x & \le\int_{-\infty}^{b-L}[\mathscr{S}_{b,\gamma_{\mathrm{R}}}[v_{\mathrm{B}},v_{\mathrm{T}}](x)-\mathscr{S}_{b,\gamma_{\mathrm{L}}}[v_{\mathrm{B}},v_{\mathrm{T}}](x)|\,\dif x\nonumber \\
 & =\int_{-\infty}^{\overline{Z}_{b}[v_{\mathrm{B}},v_{\mathrm{T}}](b-L)}[-\tanh(\zeta-\gamma_{\mathrm{R}})+\tanh(\zeta-\gamma_{\mathrm{R}})]\,\dif\zeta\nonumber \\
 & \le C\e^{2\overline{Z}_{b}[v_{\mathrm{B}},v_{\mathrm{T}}](b-L)},\label{eq:chgvar-left}
\end{align}
with a constant $C$ depending only on $\gamma_{\mathrm{L}}$ and
$\gamma_{\mathrm{R}}$. Similarly, 
\begin{equation}
\int_{b+L}^{\infty}|[v-\tilde{v}](x)|\,\dif x\le C\e^{-2\overline{Z}_{b}[v_{\mathrm{B}},v_{\mathrm{T}}](b+L)}.\label{eq:chvar-right}
\end{equation}
Using \eqref{centerpart}--\eqref{chvar-right} in \eqref{splitupL1}
yields \eqref{vvtildeL1}.
\end{proof}
Now we can prove \propref{convergenceinL1}.
\begin{proof}[Proof of \propref{convergenceinL1}.]
We set $Z_{b,t}=Z_{b,t}[(u_{\mathrm{B}},u_{\mathrm{T}})(0,\cdot)]$
and consider
\[
\mathbf{w}=(w_{\mathrm{B}},w_{\mathrm{T}})\sim\hat{\nu}_{a_{\mathrm{B}},a_{\mathrm{T}}}^{[b]}.
\]
By the Birkhoff ergodic theorem, \cite[Theorem 1.2, property (P5)]{DGR19},
and the fact that $\hat{\nu}_{a_{\mathrm{B}},a_{\mathrm{T}}}^{[b]}$
is absolutely continuous with respect to $\nu_{a_{\mathrm{B}},a_{\mathrm{T}}}$,
we have
\[
\lim_{L\to\pm\infty}\frac{1}{L}\int_{b}^{b+L}[w_{\mathrm{T}}-w_{\mathrm{B}}](x)\,\dif x=\pm(a_{\mathrm{T}}-a_{\mathrm{B}})
\]
almost surely, and in particular in probability. Hence, given $\eps>0$,
there is an $L_{\eps}<\infty$ so that~$L\ge L_{\eps}$ then
\begin{equation}
\mathbb{P}\left(\overline{Z}_{b}[\mathbf{w}](b-L)\ge-\frac{1}{2}(a_{\mathrm{T}}-a_{\mathrm{B}})L\text{ or }\overline{Z}_{b}[\mathbf{w}](b+L)\le\frac{1}{2}(a_{\mathrm{T}}-a_{\mathrm{B}})L\right)<\frac{\eps}{4}.\label{eq:Z0wL}
\end{equation}
In addition, we can choose $L_{\eps}$ so large that for all $L\ge L_{\eps}$
we have
\begin{equation}
2C\e^{-(a_{\mathrm{T}}-a_{\mathrm{B}})L}<\eps/2,\label{eq:2CeL}
\end{equation}
with $C$ as in \lemref{splitupL1}. By \propref{convergenceinL1},
we can find a $T_{\eps}<\infty$ so large that if $T\ge T_{\eps}$
and~$S_{T}\sim\Uniform([1,T+1])$ is independent of everything else,
then (using in addition \eqref{Z0wL})
\begin{align}
\mathbb{P} & \left(\begin{aligned} & \overline{Z}_{b}[\tau_{b-b_{S_{T}}}(u_{\mathrm{B}},u_{\mathrm{T}})(S_{T},\cdot)](b-L)\ge-\frac{L}{2}(a_{\mathrm{T}}-a_{\mathrm{B}})\\
 & \qquad\text{or }\overline{Z}_{b}[\tau_{b-b_{S_{T}}}(u_{\mathrm{B}},u_{\mathrm{T}})(S_{\mathrm{T}},\cdot)](b+L)\le\frac{L}{2}(a_{\mathrm{T}}-a_{\mathrm{B}})
\end{aligned}
\right)<\frac{\eps}{2}.\label{eq:Zbnottooweird}
\end{align}
and (using in addition \propref{Scts})
\begin{equation}
\mathbb{P}\left(\left\Vert \tau_{b-b_{S_{T}}}u(S_{T},\cdot)-\mathscr{S}_{b,\gamma}[\tau_{b-b_{S_{T}}}(u_{\mathrm{B}},u_{\mathrm{T}})(S_{T},\cdot)]\right\Vert _{\mathcal{C}_{\p_{\ell}}}\ge\frac{\eps}{4L^{1+\ell}}\right)<\frac{\eps}{2}.\label{eq:CpLnormnottoobig}
\end{equation}
Then we can compute, using \eqref{vvtildeL1},
\begin{align}
 & \left\Vert \tau_{b-b_{S_{T}}}u(S_{T},\cdot)-\mathscr{S}_{b,\gamma}[\tau_{b-b_{S_{T}}}(u_{\mathrm{B}},u_{\mathrm{T}})(S_{T},\cdot)]\right\Vert _{L^{1}(\mathbb{R})}\nonumber \\
 & \qquad\le2\langle|b|+L\rangle^{1+\ell}\left\Vert \tau_{b-b_{S_{T}}}u(S_{T},\cdot)-\mathscr{S}_{b,\gamma}[\tau_{b-b_{S_{T}}}(u_{\mathrm{B}},u_{\mathrm{T}})(S_{T},\cdot)]\right\Vert _{\mathcal{C}_{\p_{\ell}}}\nonumber \\
 & \qquad\qquad+C\exp\left\{ 2\overline{Z}_{b}[\tau_{b-b_{S_{T}}}(u_{\mathrm{B}},u_{\mathrm{T}})(S_{T},\cdot)](b-L)\right\} \nonumber \\
 & \qquad\qquad+C\exp\left\{ -2\overline{Z}_{b}[\tau_{b-b_{S_{T}}}(u_{\mathrm{B}},u_{\mathrm{T}})(S_{T},\cdot)](b+L)\right\} .\label{eq:splituptheL1}
\end{align}
Using \eqref{2CeL}--\eqref{CpLnormnottoobig} in \eqref{splituptheL1},
we get
\[
\mathbb{P}\left(\left\Vert \tau_{b-b_{S_{T}}}u(S_{T},\cdot)-\mathscr{S}_{b,\gamma}[\tau_{b-b_{S_{T}}}(u_{\mathrm{B}},u_{\mathrm{T}})(S_{T},\cdot)]\right\Vert _{L^{1}(\mathbb{R})}\ge\eps\right)<\eps,
\]
so 
\begin{equation}
\|\tau_{b-b_{S_{T}}}u(S_{T},\cdot)-\mathscr{S}_{b,\gamma}[\tau_{b-b_{S_{T}}}(u_{\mathrm{B}},u_{\mathrm{T}})(S_{T},\cdot)]\|_{L^{1}(\mathbb{R})}\to0\text{ in probability as }T\to\infty.\label{eq:convinprob}
\end{equation}
On the other hand, by the $L^{1}$ contractivity property (proved
as \cite[Proposition 3.2]{DGR19}), with probability~$1$ the norm
\[
\|\tau_{b-b_{t}}u(t,\cdot)-\mathscr{S}_{b,\gamma}[\tau_{b-b_{t}}(u_{\mathrm{B}},u_{\mathrm{T}})(t,\cdot)]\|_{L^{1}(\mathbb{R})}
\]
is decreasing in $t$. Together with \eqref{convinprob} this means
in fact this norm goes to zero almost surely.
\end{proof}
Now we can remove the random time $S_{T}$ in the statement of \propref{timeaveragedconvergence},
proving \thmref{shock-stability}.
\begin{proof}[Proof of \thmref{shock-stability}.]
First we note that by \propref{convergenceinL1} for each $i\in\{1,\ldots,N\}$
we have (setting $b=b^{(i)}$, $b_{t}=b_{t}^{(i)}$, and $\gamma=0$)
\begin{align*}
0 & =\lim_{t\to\infty}\|\tau_{b^{(i)}-b_{t}^{(i)}}u(t,\cdot)-\mathscr{S}_{b^{(i)},0}[\tau_{b^{(i)}-b_{t}^{(i)}}(u_{\mathrm{B}},u_{T})(t,\cdot)\|_{L^{1}(\mathbb{R})}\\
 & =\lim_{t\to\infty}\|u(t,\cdot)-\tau_{-b^{(i)}+b_{t}^{(i)}}\mathscr{S}_{b^{(i)},0}[\tau_{b^{(i)}-b_{t}^{(i)}}(u_{\mathrm{B}},u_{T})(t,\cdot)\|_{L^{1}(\mathbb{R})}\\
 & =\lim_{t\to\infty}\|u(t,\cdot)-\mathscr{S}_{b_{t}^{(i)},0}[(u_{\mathrm{B}},u_{T})(t,\cdot)\|_{L^{1}(\mathbb{R})},
\end{align*}
which is \eqref{trueL1conv}.

Let $\tilde{\mathbf{u}}=(u_{1},\ldots,u_{N})$, $\mathbf{u}=(u_{\mathrm{B}},u_{\mathrm{T}},\tilde{\mathbf{u}})$,
with notation as in the statement of the theorem. The assumption \eqref{setbi}
means that $(a_{\mathrm{B}},a_{\mathrm{T}},u_{i}(0,\cdot))\in\mathcal{X}_{\mathrm{Sh},b^{(i)},0}$.
We set $Z_{t}=Z_{b^{(1)},t}[a_{\mathrm{B}},a_{\mathrm{T}}]$ and $b=b^{(1)}$,
$b_{t}=b_{t}^{(1)}$. By the same argument as in the proof of tightness
in \propref{timeaveragedconvergence}, we see that $\{\tau_{b-b_{t}}\mathbf{u}(t,\cdot)\}_{t\ge1}$
is tight in the topology of $\mathcal{X}_{\mathrm{BT}}\times\mathcal{X}^{N}$.
Suppose that we have a sequence $t_{k}\uparrow\infty$ and a limiting
random variable $\mathbf{w}=(w_{\mathrm{B}},w_{\mathrm{T}},w_{1},\ldots,w_{N})\in\mathcal{X}_{\mathrm{BT}}\times\mathcal{X}^{N}$
so that
\[
\tau_{b-b_{t_{k}}}\mathbf{u}(t_{k},\cdot)\to\mathbf{w}
\]
in law in the topology of $\mathcal{X}_{\mathrm{BT}}\times\mathcal{X}^{N}$.
By \propref{convergetotildenu}, we have
\begin{equation}
\Law(w_{\mathrm{B}},w_{\mathrm{T}})=\hat{\nu}_{a_{\mathrm{B}},a_{\mathrm{T}}}^{[b]}.\label{eq:convergence}
\end{equation}
Therefore, using \propref{Scts}, we have
\[
\tau_{b-b_{t_{k}}}u_{i}(t_{k},\cdot)-\mathscr{S}_{b^{(i)},0}[\tau_{b-b_{t_{k}}}(u_{\mathrm{B}},u_{\mathrm{T}})(t,\cdot)]\xrightarrow[k\to\infty]{\mathrm{law}}w_{i}-\mathscr{S}_{b^{(1)},b^{(i)}-b^{(1)}}[w_{\mathrm{B}},w_{\mathrm{T}}]
\]
with respect to the topology of $\mathcal{X}$. On the other hand,
\propref{convergenceinL1} implies that, with probability $1$, for
each $1\le i\le N$
\[
\lim_{t\to\infty}\|\tau_{b-b_{t}}u_{i}(t,\cdot)-\mathscr{S}_{b^{(i)},0}[\tau_{b-b_{t}}(u_{\mathrm{B}},u_{\mathrm{T}})(t,\cdot)]\|_{L^{1}(\mathbb{R})}=0.
\]
Combined, the last two displays show that $w_{i}=\mathscr{S}_{b^{(1)},b^{(i)}-b^{(1)}}[w_{\mathrm{B}},w_{\mathrm{T}}]$
almost surely. Since the topology of weak convergence of probability
measures with respect to the topology of $\mathcal{X}_{\mathrm{BT}}\times\mathcal{X}^{N}$
is metrizable, this, \eqref{convergence}, and \propref{Scts} imply
\eqref{convinlaw}.
\end{proof}

\appendix

\section{A technical lemma\label{appendix:techlemma}}
\begin{lem}
\label{lem:inversionthing}Let $\mathcal{Y}$ be a metric space and
let $(q\mapsto F_{q}):\mathcal{Y}\to\mathcal{C}_{\mathrm{loc}}^{1}(\mathbb{R})$
be continuous and such that $\partial_{x}[F_{q}(x)]>0$ for all $q\in\mathcal{Y}$
and all $x\in\mathbb{R}$. Let $G:\mathcal{Y}\to\mathbb{R}$ be continuous.
Then the map $\mathcal{Y}\ni q\mapsto F_{q}^{-1}(G(q))\in\mathbb{R}$
is continuous.
\end{lem}

\begin{proof}
Let $q\in\mathcal{Y}$ and let $\eps>0$. There is a $\kappa>0$ so
that
\begin{equation}
\inf_{x\ :\ |x-F_{q}^{-1}(G(q))|<2\eps}F_{q}'(x)\ge\kappa.\label{eq:derivnotzero}
\end{equation}
Since $F_{q}^{-1}\circ G:\mathcal{Y}\to\mathbb{R}$ is continuous,
there is a $\delta>0$ so that if $d_{\mathcal{Y}}(q,\tilde{q})<\delta$,
then
\begin{equation}
\left|F_{q}^{-1}(G(q))-F_{q}^{-1}(G(\tilde{q}))\right|<\eps\label{eq:FinvGcts}
\end{equation}
and
\begin{equation}
\sup_{x\ :\ |x-F_{q}^{-1}(G(q))|<2\eps}|F_{\tilde{q}}(x)-F_{q}(x)|<\kappa\eps/2.\label{eq:closetog}
\end{equation}

Now if $d_{\mathcal{Y}}(q,\tilde{q})<\delta$ then $|F_{q}^{-1}(G(\tilde{q}))+\eps-F_{q}^{-1}(G(q))|<2\eps$,
so
\begin{align*}
F_{\tilde{q}} & (F_{q}^{-1}(G(\tilde{q}))+\eps)-G(\tilde{q})\\
 & =F_{\tilde{q}}(F_{q}^{-1}(G(\tilde{q}))+\eps)-F_{q}(F_{q}^{-1}(G(\tilde{q}))+\eps)+F_{q}(F_{q}^{-1}(G(\tilde{q}))+\eps)-F_{q}(F_{q}^{-1}(G(\tilde{q})))\\
 & >-\kappa\eps/2+\kappa\eps=\kappa\eps/2
\end{align*}
by \eqref{derivnotzero} and \eqref{closetog}. This means that
\[
F_{q}^{-1}(G(\tilde{q}))+\eps>F_{\tilde{q}}^{-1}(G(\tilde{q})+\kappa\eps/2)\ge F_{\tilde{q}}^{-1}(G(\tilde{q})).
\]
Similarly, we have
\[
F_{q}^{-1}(G(\tilde{q}))-\eps<F_{\tilde{q}}^{-1}(G(\tilde{q})),
\]
so in fact we have
\begin{equation}
|F_{q}^{-1}(G(\tilde{q}))-F_{\tilde{q}}^{-1}(G(\tilde{q}))|<\eps.\label{eq:changesubscript}
\end{equation}
Combining \eqref{FinvGcts} and \eqref{changesubscript}, we obtain
\[
|F_{q}^{-1}(G(q))-F_{\tilde{q}}^{-1}(G(\tilde{q}))|\le|F_{q}^{-1}(G(q))-F_{q}^{-1}(G(\tilde{q}))|+|F_{q}^{-1}(G(\tilde{q}))-F_{\tilde{q}}^{-1}(G(\tilde{q}))|<2\eps.
\]
This completes the proof.
\end{proof}
\bibliographystyle{habbrv}

\end{document}